\documentclass[a4paper,11 pt]{amsart}

\RequirePackage{amsmath, amssymb, amsthm, amsfonts}

\newtheorem{theo}{Theorem}[section]
\newtheorem{lem}[theo]{Lemma}
\newtheorem{prop}[theo]{Proposition}
\newtheorem{cor}[theo]{Corollary}
\newtheorem{defin}[theo]{Definition}
\newtheorem{notation}[theo]{Notation}
\newtheorem{alg}[theo]{Algorithm}

\theoremstyle{definition}

\newtheorem{rem}[theo]{Remark}

\newcommand{\Hom}{{\mathrm{Hom}}}
\newcommand{\PSL}{{\mathrm{PSL}}}

\newcommand{\SL}{{\mathrm{SL}}}
\newcommand{\PGL}{{\mathrm{PGL}}}
\newcommand{\GL}{{\mathrm{GL}}}
\newcommand{\Aut}{{\mathrm{Aut}}}
\newcommand{\Out}{{\mathrm{Out}}}
\newcommand{\Inn}{{\mathrm{Inn}}}
\newcommand{\ord}{{\mathrm{ord}}}
\newcommand{\sgn}{{\mathrm{sgn}}}

\newcommand{\tr}{{\mathrm{tr}}}
\newcommand{\Tr}{\mathcal{T}}

\newcommand{\Ga}{\Gamma}
\newcommand{\al}{\alpha}
\newcommand{\be}{\beta}
\newcommand{\ga}{\gamma}
\newcommand{\la}{\lambda}

\newcommand{\ZZ}{\mathbb{Z}}

\newcommand{\FQ}{\mathbb{F}_q}

\title{Beauville surfaces, moduli spaces and finite groups}
\author{Shelly Garion, Matteo Penegini}
\address{Shelly Garion\\  Max-Planck-Institute for Mathematics\\ D-53111 Bonn, Germany}
\email{shellyg@mpim-bonn.mpg.de}
\address{Matteo Penegini\\
Lehrstuhl Mathematik VIII\\ Universit\"at Bayreuth, NWII\\ D-95440
Bayreuth, Germany} \email{matteo.penegini@uni-bayreuth.de}

\begin{document}


\maketitle


\begin{abstract}
In this paper we give the asymptotic growth of the number of connected components 
of the moduli space of surfaces of general type corresponding to certain families of
Beauville surfaces with group either $\PSL(2,p)$, or an alternating group,
or a symmetric group or an abelian group.
We moreover extend these results to regular surfaces isogenous to a higher product
of curves.
\end{abstract}

\section{Introduction}\label{sect.intro}
\subsection{Beauville surfaces and surfaces isogenous to a higher product}
A surface $S$ is \emph{isogenous to a higher product of curves} if it is a
quotient $(C_1 \times C_2)/G$, where $C_1$ and $C_2$ are curves of
genus at least two, and $G$ is a finite group acting freely on
$C_1 \times C_2$.

In \cite{Cat00} it has been proved that any
surface isogenous to a higher product
has a unique \emph{minimal realization} as a quotient $(C_1 \times C_2)/G$,
where $G$ is a finite group acting freely and with the property
that no element acts trivially on one of the two factors $C_i$. We shall
then work only with minimal realizations.

We have two cases: the \emph{mixed} case where the action of $G$
exchanges the two factors (and then $C_1$ and $C_2$ are
isomorphic), and the \emph{unmixed} case where $G$ acts diagonally
on their product.

We shall use the standard notation in surface theory. We denote by
$p_g:=h^0(S,\Omega^2_S)$ the \emph{geometric genus} of $S$,
$q:=h^0(S,\Omega^1_S)$ the \emph{irregularity} of $S$,
$\chi(S)=1+p_g-q$ the \emph{holomorphic Euler-Poincar\'{e}
characteristic}, $e(S)$ the \emph{topological Euler number}, and
$K^2_S$ the \emph{self-intersection of the canonical divisor} (see e.g. \cite{Be}).

A surface $S$ isogenous to a higher product is in particular a minimal surface
of general type and the numerical invariants of $S$ are related by the following formulae
\begin{equation}\label{eq.chie} K^2_S=8\chi(S) \text{ and } e(S)=4\chi(S),
\end{equation}
by \cite[Theorem 3.4]{Cat00}. Moreover, by \cite{Cat00}, 
the irregularity of these surfaces is computed by
\begin{equation}\label{eq.irregularity} q(S)=g(C_1/G)+g(C_2/G).
\end{equation}

By the above formula \eqref{eq.irregularity}, a surface $S$ isogenous to a higher product
of curves has $q(S)=0$ if and only if the two quotients $C_i/G$
are isomorphic to $\mathbb{P}^1$. Moreover if both coverings $C_i
\rightarrow C_i/G\cong \mathbb{P}^1$ are ramified in exactly $3$
points we call $S$ a \emph{Beauville surface}. This last condition
is equivalent to saying that Beauville surfaces are rigid, i.e.
have no nontrivial deformations.

Beauville surfaces were introduced by Catanese in \cite{Cat00},
inspired by a construction of Beauville (see \cite{Be})
After this inspiring  paper the interest in Beauville surfaces has been
enormously increased, see for example \cite{BCG05, BCG06, FG, FGJ,
FJ, FMP, GLL, GP, GJT, GM}.


In this paper we shall deal only with \emph{regular} (i.e., with $q=0$) surfaces isogenous to a product of unmixed type, clearly the Beauville surfaces are among them.

Working out the definition of surfaces isogenous to a product one sees that
there is a pure group theoretical condition which characterizes
the groups of such surfaces: the existence of a so-called ''\emph{ramification structure}''
(see the discussion in Section 2).

\begin{defin}\label{defn.sphergen}
Let $G$ be a finite group and $r \geq 3$ an integer.
\begin{itemize}
\item An $r-$tuple $T=(x_1,\ldots,x_r)$ of elements of $G$ is called a
\emph{spherical $r-$system of generators} of $G$ if $ \langle
x_1,\ldots,x_r \rangle=G$ and $x_1 \cdot \ldots \cdot x_r=1$.

\item We say that $T$ is of \emph{type} $\tau=(m_1, \dots ,m_r)$ if
the orders of $(x_1,\dots,x_r)$ are respectively $(m_1, \dots
,m_r)$.

\item Moreover, two spherical $r_i-$systems $T_1=(x_1, \dots , x_{r_1})$ and $T_2=(x_1, \dots , x_{r_2})$
are said to be \emph{disjoint}, if:
\begin{equation}\label{eq.sigmasetcond} \Sigma(T_1)
\bigcap \Sigma(T_2)= \{ 1 \},
\end{equation}
where
\[ \Sigma(T_i):= \bigcup_{g \in G} \bigcup^{\infty}_{j=0} \bigcup^{r_i}_{k=1} g \cdot x^j_{i,k} \cdot
g^{-1}.
\]
\end{itemize}
\end{defin}

\begin{defin}\label{def.rami.structure} 
Let $3 \leq r_1,r_2 \in \mathbb{N}$ and let $\tau_1=(m_{1,1}, \dots ,m_{1,r_1})$,
$\tau_2=(m_{2,1}, \dots ,m_{2,r_2})$ be two sequences of natural numbers such
that $m_{i,k} \geq 2$.  

A \emph{(spherical-) unmixed ramification structure} of type $(\tau_1,\tau_2)$ and size
$(r_1,r_2)$ for a finite group $G$, is a pair $(T_1,T_2)$
of disjoint spherical systems of generators of $G$, whose types are $(\tau_1,\tau_2)$.

An \emph{unmixed Beauville structure} is an unmixed ramification structure with
$r_1=r_2=3$.
\end{defin}

\begin{defin} A triple $(r,s,t)\in \mathbb{N}^3$ is said to be hyperbolic if
$$\frac{1}{r}+\frac{1}{s}+\frac{1}{t}<1.$$
\end{defin}

\subsection{Moduli spaces of surfaces}

By a celebrated Theorem of Gieseker (see \cite{Gie}), once the two
invariants of a minimal surface $S$ of general type, the
\emph{self-intersection of the canonical divisor} $y:=K^2_S$ and
the \emph{holomorphic Euler-Poincar\'e characteristic}
$x:=\chi(S)$, are fixed, then there exists a coarse
quasiprojective moduli space $\mathcal{M}_{y,x}$ of minimal smooth
complex surfaces of general type with these invariants. This space
consists of a finite number of connected components which can have
different dimensions, see \cite{Cat84}. The union $\mathcal{M}$
over all admissible pairs of invariants ($y,x$) of these spaces is
called the \emph{moduli space of surfaces of general type}.

If $S$ is a smooth minimal surface of general type, which is also regular, 
we denote by $\mathcal{M}(S)$ the subvariety of
$\mathcal{M}_{y,x}$, corresponding to surfaces
(orientedly) homeomorphic to $S$. Moreover we shall denote by
$\mathcal{M}^0_{y,x}$ the subspace of the moduli space
corresponding to regular surfaces.

It is known that
the number of connected components $\delta(y, x)$ of
$\mathcal{M}^0_{y,x}$ is bounded from above by a function in $y$, more precisely 
it follows from \cite{Cat92} that
$\delta(y, x) \leq cy^{77y^2}$, where $c$ is a positive constant. Hence, the number of
components has an exponential upper bound in $K^2$.

There are also some results regarding the lower bound. In \cite{Man}, for example,
a sequence $S_n$ of simply connected
surfaces of general type was constructed, such that the lower bound
for the number of the connected components  $\delta(S_n)$  of $\mathcal{M}(S_n)$ is given by
\[
\delta(S_n) \geq y_n^{\frac{1}{5}log y_n}.
\]

\subsection{The results}
The motivation of this paper is to show that using pure group
theoretical methods we are able to give the asymptotic growth of
the number of connected components of the moduli space of surfaces
of general type relative to certain sequences of surfaces. More
precisely, we exploit the definition and properties of regular
surfaces isogenous to a product of curves and in particular, 
Beauville surfaces, to reduce the geometric problem
of computing the number of connected components into the algebraic one of counting
orbits of some group actions, which can be effectively computed.

In \cite{Cat00}, Catanese studied the moduli space of surfaces
isogenous to a higher product of curves (see \cite[Theorem 4.14]{Cat00}). As a
result, he obtains that the moduli space of regular surfaces isogenous to
a higher product with fixed invariants: a finite group $G$ and
types $(\tau_1,\tau_2)$, consists of a
finite number of connected components of $\mathcal{M}$. We remark
here that since Beauville surfaces are rigid, a Beauville surface
yields an isolated point in the moduli space. A group theoretical
method to count the number of these components was given in
\cite{BC}. Using this method, we deduce the following Theorems, 
in which we use the following standard notations.

\begin{notation} Denote:
\begin{itemize}
\item $h(n)=O(g(n))$, if $h(n) \leq cg(n)$ for some 
constant $c>0$, as $n \to \infty$.

\item $h(n)=\Omega(g(n))$, if $h(n) \geq cg(n)$ for some 
constant $c>0$, as $n \to \infty$.

\item $h(n)=\Theta(g(n))$, if $c_1g(n) \leq h(n) \leq c_2g(n)$ for
some constants $c_1,c_2>0$, as $n \to \infty$.
\end{itemize}
\end{notation}

We shall first consider alternating and symmetric groups.

\begin{theo}\label{thm.moduli.An}
Let $\tau_1=(m_{1,1}, \dots ,m_{1,r_1})$ and $\tau_2=(m_{2,1}, \dots ,m_{2,r_2})$ be two sequences of natural numbers such
that $m_{i,k} \geq 2$ and $\sum_{k=1}^{r_i}(1-1/m_{i,k}) > 2$ for
$i=1,2$. Let $h(\mathfrak{A}_n; \tau_1, \tau_2)$ be the number
of connected components of the moduli space of surfaces isogenous to a product with $q=0$, 
with the alternating group $\mathfrak{A}_n$, and with type
$(\tau_1,\tau_2)$. Then
\[
    (a) \quad h(\mathfrak{A}_n; \tau_1, \tau_2) = \Omega(n^{r_1+r_2}),
\]
and moreover,
\[
    (b) \quad h(\mathfrak{A}_n; \tau_1, \tau_2) =
    \Omega\bigl(\bigl(\log(\chi)\bigr)^{r_1+r_2-\epsilon}\bigr).
\]
where $0< \epsilon \in \mathbb{R}$.
\end{theo}

\begin{theo}\label{thm.moduli.Sn}
Let $\tau_1 = (m_{1,1},\dots,m_{1,r_1})$ and $\tau_2 =
(m_{2,1},\dots,m_{2,r_2})$ be two sequences of natural numbers such
that $m_{i,k} \geq 2$, at least two of $(m_{i,1},\dots,m_{i,r_i})$
are even and $\sum_{k=1}^{r_i}(1-1/m_{i,k}) > 2$, for $i=1,2$. 
Let $h(\mathfrak{S}_n; \tau_1, \tau_2)$ be the number
of connected components of the moduli space of surfaces isogenous to a product with $q=0$, 
with the symmetric group
$\mathfrak{S}_n$, and with type $(\tau_1,\tau_2)$. Then
\[
    (a)\quad h(\mathfrak{S}_n; \tau_1, \tau_2) = \Omega(n^{r_1+r_2}),
\]
and moreover,
\[
    (b) \quad h(\mathfrak{S}_n; \tau_1, \tau_2) =
    \Omega\bigl(\bigl(\log(\chi)\bigr)^{r_1+r_2-\epsilon}\bigr).
\]
where $0< \epsilon \in \mathbb{R}$.
\end{theo}

The proofs of part (a) of both Theorems are presented in Section~\ref{subsec.Hur.An.Sn}, 
and are based on results of Liebeck and Shalev \cite{LS04}. The proofs of part (b)
of both Theorems appear in Section \ref{sec.count.moduli}.
In the special case of Beauville surfaces we immediately deduce the following.

\begin{cor}\label{cor.moduli.Beu.An}
Let $\tau_1=(r_1,s_1,t_1)$ and $\tau_2=(r_2,s_2,t_2)$ be two
hyperbolic types and let $h(\mathfrak{A}_n; \tau_1, \tau_2)$ be the number of
Beauville surfaces with the alternating group $\mathfrak{A}_n$ and with types
$(\tau_1,\tau_2)$. Then
\[
    (a) \quad h(\mathfrak{A}_n; \tau_1, \tau_2) = \Omega(n^6),
\]
and moreover,
\[
    (b) \quad h(\mathfrak{A}_n; \tau_1, \tau_2) =
    \Omega\bigl(\bigl(\log(\chi)\bigr)^{6-\epsilon}\bigr).
\]
where $0< \epsilon \in \mathbb{R}$.
\end{cor}

\begin{cor}\label{cor.moduli.Beu.Sn}
Let $\tau_1=(r_1,s_1,t_1)$ and $\tau_2=(r_2,s_2,t_2)$ be two
hyperbolic types, assume that at least two of $(r_1,s_1,t_1)$ are
even and at least two of $(r_2,s_2,t_2)$ are even, and let $h(\mathfrak{S}_n;
\tau_1, \tau_2)$ be the number of Beauville surfaces with the symmetric group
$\mathfrak{S}_n$ and with types $(\tau_1,\tau_2)$. Then
\[
    (a)\quad h(\mathfrak{S}_n; \tau_1, \tau_2) = \Omega(n^6),
\]
and moreover,
\[
    (b) \quad h(\mathfrak{S}_n, \tau_1, \tau_2) =
    \Omega\bigl(\bigl(\log(\chi)\bigr)^{6-\epsilon}\bigr).
\]
where $0< \epsilon \in \mathbb{R}$.
\end{cor}

For the group $\PSL(2,p)$ we obtain the following.

\begin{theo}\label{thm.moduli.PSLp}
Let $\tau_1$ and $\tau_2$ be two hyperbolic triples, let $p$ be an odd
prime, and consider the group $\PSL(2,p)$. Let $h(\PSL(2,p); \tau_1,
\tau_2)$ be the number of Beauville surfaces with group $\PSL(2,p)$
and with types $(\tau_1,\tau_2)$. Then there exists a constant
$c=c(\tau_1,\tau_2)$, which depends only on the types and not on
$p$, such that
\[
    \quad h(\PSL(2,p); \tau_1, \tau_2) \leq c.
\]
\end{theo}

The proof of this Theorem is presented in Section~\ref{subsec.Hu.psl} 
(see also Remark \ref{rem.PSL2p}).

The following is a natural generalization of the results of \cite{BCG05}
regarding abelian groups.

\begin{theo}\label{theo.poly.large}
Let $\{S_p\}$ be the family of surfaces isogenous to a product with $q=0$ and with group
$G_p:=(\mathbb{Z}/p \mathbb{Z})^r$ admitting a ramification structure of type $\tau_p =
(p,\dots, p)$ ($p$ appears $(r+1)-$times) where $p$ is prime. If we denote by $h(G_p;\tau_p,\tau_p)$ the number of connected components of the  moduli space of isomorphism classes of surfaces isogenous to a product
with $q=0$ admitting these data, then
\[
h(G_p;\tau_p,\tau_p)=\Theta(\chi^{r}(S_p)).
\]
 \end{theo}

Therefore, there exist families of surfaces such that the degree of
the polynomial $h$ in $\chi$ (and so in $K^2$) can be arbitrarily large.
The proof of this Theorem appears in
Section \ref{sec.count.moduli}.
In the special case of Beauville surfaces we immediately deduce the following.

\begin{cor}\label{cor.poly.beau}
Let $\{S_p\}$ be the family of Beauville surfaces with
$G_p:=(\mathbb{Z}/p \mathbb{Z})^2$ admitting a ramification structure of type $\tau_p =
(p,p,p)$ where $p \geq 5$ is prime. If we denote by $h(G_p;\tau_p,\tau_p)$ the number of Beauville surfaces admitting these data, then
\[
h(G_p;\tau_p,\tau_p)=\Theta(\chi^{2}(S_p)).
\]
 \end{cor}

\medskip

\textbf{Acknowledgement.} 
The authors are grateful to Fritz Grunewald for inspiring and motivating us to work on this problem together. He is deeply missed.

The authors would like to thank Ingrid Bauer and Fabrizio Catanese for
suggesting the problems, for many useful discussions and for their helpful
suggestions.
We are grateful to Moshe Jarden and Martin Kassabov for interesting
discussions. 
We also would like to thank G. Jones, G. Gonzales-Diez and D. Torres-Teigell 
for pointing out some subtleties in our first draft.

The authors acknowledge the support of the DFG Forschergruppe 790
''Classification of algebraic surfaces and compact complex manifolds''. 
The first author acknowledges the support of the
European Post-Doctoral Institute and the Max-Planck-Institute for
Mathematics in Bonn.

\section{From Geometry to Group Theory and Back}\label{sec.from.geometry}
\subsection{Ramification structures and Hurwitz components}
The study of surfaces isogenous to a higher product is
strictly linked to the study of branched coverings of complex
curves. We shall recall Riemann's existence theorem which translates the 
geometric problem of constructing
branch coverings into a group theoretical problem.

\begin{defin}\label{defn.orbifold.fuchsian}
Let $g',m_1, \dots , m_r$ be positive integers. An \emph{orbifold
surface group} of type $(g' \mid m_1, \dots , m_r)$ is a group
presented as follows:
\begin{multline*}  \Gamma(g' \mid m_1, \dots , m_r):=\langle a_1,b_1, \ldots , a_{g'},b_{g'},c_{1},
\ldots , c_{r} | \\
c^{m_1}_{1}=\dots=c^{m_r}_{r}=\prod_{k=1}^{g'} [a_{k},b_{k}]c_{1}
\cdot \ldots \cdot c_{r} =1 \rangle.
\end{multline*}
If $g'=0$ it is called a \emph{polygonal group}, if $g'=0$ and
$r=3$ it is called a \emph{triangle group}.
\end{defin}
We remark that an orbifold surface group is in particular cases a
\emph{Fuchsian group} (see e.g. \cite{Br} and \cite{LS04}).

The following is a reformulation of {\em Riemann's existence theorem}:

\begin{theo} A finite group $G$ acts as a group of automorphisms on some
compact Riemann
surface $C$ of genus $g$ if and only if there are natural numbers
$g', m_1, \ldots ,
m_r$, and an orbifold homomorphism
\begin{equation}\label{eq.theta}
\theta \colon \Gamma(g' \mid m_1, \dots , m_r)
\rightarrow G
\end{equation}
such that
$ord(\theta(c_i))=m_i$ $\forall i$ and such that the Riemann - Hurwitz relation holds:
\begin{equation}\label{eq.RiemHurw} 2g - 2 = |G|\left(2g'-2 + \sum_{i=1}^r \left(1 -
\frac{1}{m_i}\right)\right).
\end{equation}

\end{theo}

If this is the case, then $g'$ is the genus of $C':=C/G$. The
$G$-cover $C \rightarrow
C'$ is branched in $r$ points $p_1, \ldots , p_r$ with branching
indices $m_1, \ldots ,
m_r$, respectively.

We obtain that the datum of a surface isogenous to a higher product
of unmixed type $S=(C_1 \times C_2)/G$ with $q=0$ is determined,
once we look at the monodromy of each covering of $\mathbb{P}^1$, by
the datum of a finite group $G$ together with two respective
disjoint spherical $r_i-$systems of generators $T_1:=(x_1, \dots ,
x_{r_1})$ and $T_2:=(x_1, \dots , x_{r_2})$, such that the types of
the systems satisfy~\eqref{eq.RiemHurw} with $g'=0$ and respectively
$g=g(C_i)$. The condition of being disjoint ensures that the action
of $G$ on the product of the two curves $C_1 \times C_2$ is free. We
remark here that this can be specialized to $r_i=3$, and therefore
can be used to construct Beauville surfaces. This description
gives at once the Definition \ref{def.rami.structure}.
\begin{rem} Note that a group $G$ and an unmixed ramification structure
(or equivalently a Beauville structure) for it determine the main
invariants of the surface $S$. Indeed, as a consequence of the
Zeuthen-Segre formula one has:

\begin{equation}\label{eq.SZ} e(S)=4\frac{(g(C_1)-1)(g(C_2)-1)}{|G|}.
\end{equation}

Hence, by~\eqref{eq.chie} and~\eqref{eq.RiemHurw} we obtain:
\begin{equation}\label{eq.pginfty}
4\chi(S)=4(1+p_g)=|G|\cdot\left({-2+\sum^{r_1}_{k=1}(1-\frac{1}{m_{1,k}}))}\right)
\cdot\left({-2+\sum^{r_2}_{k=1}(1-\frac{1}{m_{2,k}}))}\right),
\end{equation}
\end{rem}
and so, in the Beauville case,
\[
    4\chi(S)=4(1+p_g)=|G|(1-\mu_1)(1-\mu_2),
\]
where
\begin{equation}
\label{eq.RHtre} \mu_i:=
\frac{1}{m_{i,1}}+\frac{1}{m_{i,2}}+\frac{1}{m_{i,3}}, \quad
(i=1,2).
\end{equation}

Now, it is left to verify that indeed $g(C_1)\geq 2$ and $g(C_2)\geq
2$. This follows from Equation~\eqref{eq.RiemHurw} and from the
following Lemma.

\begin{lem}\label{lem.genus.ge.2}
Let $G$ be a finite non-trivial group and $(T_1,T_2)$ a 
spherical unmixed ramification structure of $G$ of size $(r_1,r_2)$, then
\begin{equation}\label{eq.Rim.Hur.Condition.go}
\mathbb{Z} \ni \frac{|G|
(-2+\sum^{r_i}_{l=1}(1-\frac{1}{m_{i,l}}))}{2}+1 \geq 2, \text{  for
$i=1,2$}.
\end{equation}
\end{lem}

The fact that the number in~\eqref{eq.Rim.Hur.Condition.go} is an
integer follows from Riemann's existence theorem. We need to prove
that this integer is at least $2$, namely that
$\sum^{r_i}_{l=1}(1-\frac{1}{m_{i,l}}) > 2$ for $i=1,2$. We shall
give two proofs for this fact, a geometric one and a group theoretic
one (in Section~\ref{sect.extension}), 
both are based on results of Bauer, Catanese and Grunewald.

\begin{proof}[Geometrical proof]
Let $S=(C_1 \times C_2) / G$ be a surface isogenous to a product
with $q(S)=0$, notice first that  $g(C_1) \neq 1$.

Indeed, suppose that $g(C_1)=1$, then $S \rightarrow C_2/G \cong
\mathbb{P}^1$ is an elliptic fibration with fibre isomorphic to
$C_1$ or to a multiple of $C_1$. Since $C_1$ is an elliptic curve,
the Zeuthen-Segre Theorem holds in the following form:
\[
e(S)=4\bigl(g(C_1)-1\bigr)\bigl(g(C_2/G)-1\bigr)=0.
\]
Since $S$ is isogenous to a product $4\chi(S)=e(S)=0$, but we have
$\chi(S)=1+p_g-q=1+p_g>0$. Hence $g(C_1) \neq 1$.

Second, suppose that $S$ is a $\mathbb{P}^1$-bundle. Then $S$ cannot
be non-rational, because non-rational ruled surfaces have $q>0$.
Hence $S$ must be rational. If $S$ is rational then $p_g=0$, and
surfaces with $p_g=q=0$ isogenous to a product were classified by
Bauer-Catanese-Grunewald in~\cite{BCG08}, and the only rational one
is $S=\mathbb{P}^1 \times \mathbb{P}^1$, therefore $G$ is trivial and
this case is also excluded.
\end{proof}


Let $S$ be a surface isogenous to a higher product of unmixed type with $q=0$,
and with group $G$ and a pair of two disjoint spherical $r_i-$systems of generators of types
$(\tau_1,\tau_2)$. By~$\eqref{eq.pginfty}$ we have
$\chi(S)=\chi(G,(\tau_1,\tau_2))$, and consequentially,
by~\eqref{eq.chie}, $K^2_S=K^2(G,(\tau_1,\tau_2))=8\chi(S)$.

Let us fix a group $G$ and a pair of unmixed ramification types
$(\tau_1,\tau_2)$, and denote by $\mathcal{M}_{(G,(\tau_1,\tau_2))}$
the moduli space of isomorphism classes of surfaces isogenous to a product
admitting these data, by
\cite{Cat00} the space $\mathcal{M}_{(G,(\tau_1,\tau_2))}$ consists
of a finite number of connected components. Indeed, there is a group
theoretical procedure to count these components, which is described in
\cite{BC}.

\begin{defin} \label{def.braid.action}
The braid group of the sphere
$\mathbf{B}_r:=\pi_0(Diff(\mathbb{P}^1-\{p_1, \dots, p_r\}))$
operates on the epimorphism $\theta$ defined in~\eqref{eq.theta}:
\[
\pi_1(\mathbb{P}^1 - \{p_1, \dots,p_r\})/ \langle \gamma^{m_1},
\dots ,\gamma^{m_r}\rangle \cong \Gamma:=\Gamma(0 \mid m_1, \dots ,
m_r) \stackrel{\theta}{\longrightarrow} G.
\]
Indeed, if $\sigma \in \mathbf{B}_r$ then the operation is given
by $\theta \circ \sigma$. The orbits of this action are called
\emph{Hurwitz equivalence classes} of the spherical systems of
generators.
\end{defin}

\begin{defin}
Let $G$ be a finite group, let $r \geq 3$ and $2 \leq m_1 \leq m_2 \leq \dots \leq m_r$ be integers.
Assume that $T=(x_1,\ldots,x_r)$ is a spherical $r-$system of generators of $G$.
\begin{itemize}
\item We say that $T$ has an \emph{unordered type} $\tau$ if the
orders of $(x_1,\dots,x_r)$ are $(m_1, \dots ,m_r)$ up to a
permutation, namely, if there is a permutation $\pi \in \mathfrak{S}_r$ such
that
\[
    \ord(x_1) = m_{\pi(1)},\dots,\ord(x_r)= m_{\pi(r)}.
\]

\item We shall denote:
\[ \mathcal{S}(G,\tau):=\{\textrm{spherical $r-$systems for } G \textrm{ of type } \tau\}.
\]
\end{itemize}
\end{defin}

\begin{notation}
Let $(T_1,T_2)$ be a pair of disjoint spherical $r_i-$systems of
generators of type $(\tau_1,\tau_2)$, we call the pair $(T_1,T_2)$
\emph{unordered} if $T_1$ and $T_2$ have unordered types $\tau_1$
and $\tau_2$ respectively.

We shall denote by $\mathcal{U}(G;\tau_1,\tau_2)$ the set of all
unordered pairs $(T_1,T_2)$ of disjoint spherical $r_i-$systems of
generators of type $(\tau_1,\tau_2)$.
\end{notation}

\begin{theo}\cite[Theorem 1.3]{BC}.
Let $S$ be a surface isogenous to a higher product
of unmixed type and with $q=0$.
Then to $S$ we attach its finite group $G$ (up to isomorphism) and
the equivalence classes of an unordered pair of disjoint spherical
systems of generators $(T_1,T_2)$ of  $G$, under the equivalence
relation generated by:
\begin{enumerate}\renewcommand{\theenumi}{\it \roman{enumi}}
\item  Hurwitz equivalence for $T_1$;

\item  Hurwitz equivalence for $T_2$;

\item Simultaneous conjugation for $T_1$ and $T_2$, i.e., for
$\phi \in \Aut(G)$ we let $\bigl(T_1:=(x_{1,1},\dots , x_{r_1,1}),
\ \ \ T_2:=(x_{1,2},\dots, x_{r_2,2})\bigr)$ be equivalent to
\[\bigl(\phi(T_1):=(\phi(x_{1,1}),\dots ,\phi( x_{r_1,1})),
\ \ \phi(T_2):=(\phi(x_{1,2}),\dots, \phi(x_{r_2,2}))\bigr).
\]
\end{enumerate}
Then two surfaces $S$, $S'$ are deformation equivalent if and only
if the corresponding equivalence classes of pairs of spherical
generating systems of $G$ are the same.
\end{theo}

Recall that

\begin{lem}\cite[Lemma 9.4]{Vol}.\label{lem.Vol} 
The inner automorphism group, $\Inn (G)$, leaves each braid orbit invariant.
\end{lem}

This Lemma allows us to use the above Theorem of Bauer and Catanese also for non-abelian groups, 
although the original statement was given only for abelian groups.

Once we fix a finite group $G$ and a pair of types
$(\tau_1,\tau_2)$ (of size ($r_1$,$r_2$)) of an unmixed ramification
structure for $G$, counting the number of connected components of
$\mathcal{M}_{(G,(\tau_1,\tau_2))}$ is then equivalent to the group
theoretical problem of counting the number of classes of pairs of
spherical systems of generators of $G$ of type $(\tau_1,\tau_2)$
under the equivalence relation given in the following definition
(see e.g.~\cite[\S 1.1]{BCG08}).

\begin{defin}\label{def.Hur.comp}
Denote by $h(G;\tau_1,\tau_2)$ the number of \emph{Hurwitz
components}, namely the number of orbits of
$\mathcal{U}(G;\tau_1,\tau_2)$ under the following actions:
\begin{description}
\item[if $\tau_1 \neq \tau_2$] the action of
$(\mathbf{B}_{r_1} \times \mathbf{B}_{r_2}) \times \Aut(G)$, given
by:
\[
    ((\gamma_1, \gamma_2), \phi) \cdot (T_1, T_2) := \bigr(\phi(\gamma_1(T_1)),
    \phi(\gamma_2(T_2))\bigl),
\]
where $\gamma_1 \in \mathbf{B}_{r_1}$, $\gamma_2 \in
\mathbf{B}_{r_2}$, $\phi \in \Aut(G)$ and $(T_1,T_2) \in
\mathcal{U}(G;\tau_1,\tau_2)$.

\item[if $\tau_1=\tau_2$] the action of
$(\mathbf{B}_{r} \wr \ZZ/2\ZZ) \times \Aut(G)$, 
where $\ZZ/2\ZZ$ acts on $(T_1,T_2)$ by exchanging the two factors.
\end{description}
\end{defin}

In case of Beauville surfaces we define $h$ as above substituting
$r_1$ and $r_2$ with $3$.

\subsection{Counting connected components in the moduli space}\label{sec.count.moduli}
In this Section we prove Theorems~\ref{thm.moduli.An}(b),
\ref{thm.moduli.Sn}(b) and ~\ref{theo.poly.large}.
We start by proving a more general statement regarding families of finite groups.

\begin{prop}\label{prop.asy.comp} Fix $r_1$ and $r_2$ in $\mathbb{N}$.
Let $\{G_n\}^{\infty}_{n=1}$ be a family of finite groups, which
admit an unmixed ramification structure of size $(r_1,r_2)$. Let
$\tau_{n,1} = (m_{n,1,1},\dots,m_{n,1,r_1})$ and $\tau_{n,2} =
(m_{n,2,1},\dots,m_{n,2,r_2})$ be sequences of types
$(\tau_{n,1},\tau_{n,2})$ of unmixed ramification structures for
$G_n$, and $\{S_n\}^{\infty}_{n=1}$ be the family of surfaces
isogenous to higher product with $q=0$ admitting the given data,
then as $|G_n| \stackrel{n \rightarrow \infty}{\longrightarrow}
\infty$ :
\begin{enumerate}\renewcommand{\theenumi}{\it \roman{enumi}}
    \item $\chi(S_n) = \Theta(|G_n|)$.
    \item $h(G_n; \tau_{n,1}, \tau_{n,2}) = O(\chi(S_n)^{r_1+r_2-2})$.
\end{enumerate}
\end{prop}

\begin{proof}
\begin{enumerate}\renewcommand{\theenumi}{\it \roman{enumi}}
\item  Note that, for $i=1,2$,
\[
  \frac{1}{42} \leq -2 + \sum_{j=1}^{r_i} \bigl( 1-\frac{1}{m_{n,i,j}}
  \bigr) \leq r_i-2.
\]
Indeed, for $r_i=3$, the minimal value for $(1-\mu_i)$ is $1/42$.
For $r_i=4$, the minimal value for $\bigl( -2 + \sum_{j=1}^{r_i}
\bigl( 1-\frac{1}{m_{n,i,j}} \bigr) \bigr)$ is $1/6$, and when
$r_i \geq 5$, this value is at least $1/2$.

Now, by Equation~\eqref{eq.pginfty},
\[
    4\chi(S_n) = |G_n| \cdot \left(-2 + \sum_{j=1}^{r_1} \bigl( 1-\frac{1}{m_{n,1,j}}
  \bigr)\right) \cdot \left( -2 + \sum_{j=1}^{r_2} \bigl( 1-\frac{1}{m_{n,2,j}}
  \bigr) \right),
\]
hence
\[
    \frac{|G_n|}{4 \cdot 42^2} \leq \chi(S_n) \leq \frac{(r_1-2)(r_2-2)|G_n|}{4}.
\]

\item For $i=1,2$, any spherical $r_i-$system
of generators $T_{n,i}$ contains at most $r_i-1$ independent elements of $G_n$.
Thus, the size of the set of all unordered pairs of type
$(\tau_{n,1},\tau_{n,2})$ is bounded from above, by
\[
    |\mathcal{U}(G_n;\tau_{n,1},\tau_{n,2})| \leq |G_n|^{r_1+r_2-2},
\]
and so, the number of connected components is bounded from above
by
\[
    h(G_n;\tau_{n,1},\tau_{n,2}) \leq |G_n|^{r_1+r_2-2}.
\]
Now, the result follows from (\textit{i}).
\end{enumerate}
\end{proof}

By taking $r_1=r_3=3$ we get the following Corollary.

\begin{cor}\label{cor.asy.beau}
Let $\{G_n\}^{\infty}_{n=1}$ be a family of finite groups, which
admit an unmixed Beauville structure. Let $\tau_{n,1} =
(m_{n,1,1},m_{n,1,2},m_{n,1,3})$ and $\tau_{n,2} =
(m_{n,2,1},m_{n,2,2},m_{n,2,3})$ be sequences of types
$(\tau_{n,1},\tau_{n,2})$ of unmixed Beauville structures for $G_n$,
and let $\{S_n\}^{\infty}_{n=1}$ be the family of Beauville surfaces
admitting the given data, then as $|G_n| \stackrel{n \rightarrow
\infty}{\longrightarrow} \infty$ :
\begin{enumerate}\renewcommand{\theenumi}{\it \roman{enumi}}
    \item $\chi(S_n) = \Theta(|G_n|)$.
    \item $h(G_n;\tau_{n,1},\tau_{n,2}) = O(\chi(S_n)^{4})$.
\end{enumerate}
\end{cor}

With the calculation done in this paper we can give a more
accurate description of the asymptotic growth of $h$ in case of
Beauville surfaces and surfaces isogenous to a higher product with
$q=0$, for certain families of finite groups.

\begin{rem}\label{rem.PSL2p}
Let $\tau_{1}=(r,r,r)$ and $\tau_{2}=(s,s,s)$ be two hyperbolic
types, where $r$ and $s$ are two distinct primes which are strictly
larger than $5$. By \cite{GP}, there exist
infinitely many primes $p$ for which the group $\PSL(2,p)$ admits an
unmixed Beauville structure of type $(\tau_1,\tau_2)$.

Let us consider the corresponding Beauville surfaces $S_p$, then by
Proposition~\ref{prop.asy.comp}, as $p \rightarrow \infty$:
 \[ \chi(S_p) = \Theta(p^3),
\]
while, by Theorem~\ref{thm.moduli.PSLp} we have
\[
    h(\PSL(2,p),\tau_{1},\tau_{2}) <c,
\]
where $c=c(\tau_1,\tau_2)$ is a constant which depends only on the
types and not on $p$. Namely, there exists an infinite family of
surfaces for which $h$ remains bounded while $\chi$ grows to
infinity.
\end{rem}

On the other hand, when considering the groups $\mathfrak{A}_n$ and $\mathfrak{S}_n$ one
obtains the following lower bound.

\begin{proof}[Proof of Theorem \ref{thm.moduli.An}(b) and \ref{thm.moduli.Sn}(b)]
Let $\{S_n\}$ be a
family of surfaces isogenous to a higher product
with $q=0$, with group either $\mathfrak{A}_n$ or $\mathfrak{S}_n$, and $\tau_1$
and $\tau_2$ are two types which satisfy the assumptions of the Theorems.
Then by
Proposition~\ref{prop.asy.comp}, as $n \rightarrow \infty$:
 \[
    \chi(S_n) = \Theta(n!),
\]
while, by Theorems~\ref{thm.moduli.An}(a)
and~\ref{thm.moduli.Sn}(a),
\begin{align*}
    h(\mathfrak{A}_n,\tau_{1},\tau_{2}) &= \Omega(n^{r_1+r_2}), \text{ and } \\
    h(\mathfrak{S}_n,\tau_{1},\tau_{2}) &= \Omega(n^{r_1+r_2}).
\end{align*}
Therefore,
\begin{align*}
    h(\mathfrak{A}_n,\tau_{1},\tau_{2}) &= \Omega\bigl(\bigl(\log(\chi(S_n))\bigr)^{r_1+r_2-\epsilon}\bigr)
    \text{ and }\\
    h(\mathfrak{S}_n,\tau_{1},\tau_{2}) &=
    \Omega\bigl(\bigl(\log(\chi(S_n))\bigr)^{r_1+r_2-\epsilon}\bigr).
\end{align*}

where $0< \epsilon \in \mathbb{R}$.
\end{proof}

For abelian groups one can moreover obtain both a lower and an upper bound.

\begin{proof}[Proof of Theorem~\ref{theo.poly.large}]
Consider the family $\{ S_p \}$ of surfaces isogenous to a higher product
with $q=0$, where $p$ is prime, admitting type $\tau_p =
(p,\dots, p)$ ($p$ appears $(r+1)-$times) and group
$G_p:=(\mathbb{Z}/p \mathbb{Z})^r$, then by Proposition
\ref{prop.asy.comp}, we have as $p \rightarrow \infty$:
 \[ \chi(S_p) = \Theta(p^r),
 \]
 while by Proposition \ref{prop.ZpZ.mult},
 \[ h(G_p;\tau_p,\tau_p)=\Theta(p^{r^2}).
 \]
Therefore,
\[
h(G_p;\tau_p,\tau_p)=\Theta(\chi^{r}(S_p)).
\]
\end{proof}

\section{Finite Groups, Ramification Structures and Hurwitz Components}
\subsection{Braid group actions}\label{subsec.braid.action}
Recall that the \emph{braid group} $\mathbf{B}_r$ on $r$ strands
can be presented as
\[
    \mathbf{B}_r = \langle \sigma_1,\dots,\sigma_{r-1} | \sigma_i \sigma_{i+1}
    \sigma_i = \sigma_{i+1} \sigma_i\sigma_{i+1}, \sigma_i \sigma_j =
    \sigma_j \sigma_i \text{ if } |i-j|\geq 2 \rangle.
\]

The action of $\mathbf{B}_r$ on the set of spherical $r-$systems
of generators for $G$ of unordered type $\tau=(m_1,\dots,m_r)$,
which was given in Definition~\ref{def.braid.action}, is given by
\[
    \sigma_i: (x_1,\dots,x_i,\dots,x_r) \rightarrow (x_1,\dots,x_{i-1},x_i x_{i+1}
    x_i^{-1},x_i,x_{i+2}\dots,x_r),
\]
for $i=1,\dots,r-1$.

There is also a natural action of $\Aut(G)$ given by
\[
    \phi(x_1,\dots,x_r) = (\phi(x_1),\dots,\phi(x_r)), \quad \phi
    \in \Aut(G).
\]

Since the two actions of $\mathbf{B}_r$ and $\Aut(G)$ commute, one
gets a double action of $\mathbf{B}_r \times \Aut(G)$ on the set
of spherical $r-$systems of generators for $G$ of an unordered
type $\tau=(m_1,\dots,m_r)$.

Let $x\in G$ and denote by $C=x^{\Aut(G)}$ the $\Aut(G)-$equivalence
class of $x$. Since all the elements in $C$ have the same order, we
may define $\ord(C):=\ord(x)$.

Let $\mathbf{C}=(C_1,\dots,C_r)$ be a set of $\Aut(G)$-equivalence
classes. Denote
\begin{multline*}
    N(\mathbf{C}):= \{(x_1,\dots,x_r) \in G^r:
    x_1\cdot \ldots \cdot x_r=1, \langle x_1,\dots,x_r\rangle = G \text{ and } \\
    \text{there is a permutation } \pi \in \mathfrak{S}_r \text{ with }
    x_{\pi(i)} \in C_i \text{ for all } i\}.
\end{multline*}

We say that $\mathbf{C}$ has \emph{type}
$\tau=(m_1,\dots,m_r)$ if $N(\mathbf{C}) \neq \emptyset$ and
$\ord(C_i)=m_i$ (for $i=1,\dots,r$).
$\mathbf{C}$ has an \emph{unordered type} $\tau$ if
$N(\mathbf{C}) \neq \emptyset$ and the
orders of $C_1,\dots,C_r$ are $m_1,\dots,m_r$ up to a permutation.
We denote
\[
    N(\tau)= \{\mathbf{C}=(C_1,\dots,C_r):
    \mathbf{C} \text{ has an unordered type } \tau \}.
\]

Observe that the action of $\mathbf{B}_r$ preserves the set
$N(\mathbf{C})$, since it preserves the conjugacy classes, and hence
the $\Aut(G)-$equivalence classes, of the elements in a spherical
$r-$system of generators of $G$. It is clear that the action of
$\Aut(G)$ also preserves the set $N(\mathbf{C})$. The following
Lemma immediately follows.

\begin{lem} \label{lem.two.Br.action}
Let $\tau_1$ and $\tau_2$ be two types, then
\begin{multline*}
    h(G;\tau_1,\tau_2) \geq \#\{\mathbf{C}_i,\mathbf{D}_j:
    \mathbf{C}_i=(C_{i,1},\dots,C_{i,r_1}) \text{ and }
    \mathbf{D}_j=(D_{j,1},\dots,D_{j,r_2}), \\
    \text{where } \mathbf{C}_i \in N(\tau_1) \text{ for all i, }
    \mathbf{D}_j \in N(\tau_2) \text{ for all j, } \text {and }  \\
    \text{ the } \Aut(G)-\text{classes }
    \{C_{i,k}\}_{i,k} \text{ and } \{D_{j,l}\}_{j,l}
    \text{ are all distinct} \} .
\end{multline*}
\end{lem}

One can moreover restrict the action of $\Aut(G)$ to the action of $\Inn(G) \cong G$,
which is given by
\[
   g: (x_1,\dots,x_r) \mapsto (g x_1 g^{-1},\dots,g x_r g^{-1}), \quad g \in G.
\]

Denote by $O^G_T$ the orbit of $T=(x_1,\dots,x_r)$ under the action
of $\Inn(G) \cong G$, and by $O^{\mathbf{B}_r}_T$ the orbit of $T$
under the action of $\mathbf{B}_r$. By Lemma~\ref{lem.Vol}, the action of $\Inn(G)$ leaves the
orbit $O^\mathbf{B}_T$ invariant, namely,

\begin{lem}\label{lem.Br.Inn}
$O^G_T \subseteq O^{\mathbf{B}_r}_T$.
\end{lem}
Hence, we get an induced action of $\Out(G) = \Aut(G)/\Inn(G)$ on
the set of spherical $r-$systems of generators for $G$ of unordered
type $\tau$.

In the special case of $\mathbf{B}_3$, the braid group on $3$
strands, we show that the inverse inclusion also holds, and one can
therefore deduce a more accurate bound on the number of orbits.

Let $T=(x,y,(xy)^{-1})$ be a spherical $3-$system of generators for $G$,
and let
\begin{multline*}
    T^{un}:= (x,y,(xy)^{-1}) \cup
        (y,x,(yx)^{-1}) \cup (x,(yx)^{-1},y) \\ \cup (y,(xy)^{-1},x)
        \cup ((xy)^{-1},x,y) \cup ((yx)^{-1},y,x),
\end{multline*}
be an \emph{unordered} spherical $3-$system of generators.

Observe that $O^G_{T^{un}} := \{O^G_T: T \in T^{un}\}$, where
\[
    O^G_T = \{(gxg^{-1},gyg^{-1},g(xy)^{-1}g^{-1}): g \in G \}.
\]

\begin{lem}\label{lem.B3.action}
The action of $\mathbf{B}:=\mathbf{B}_3$ preserves $O^G_{T^{un}}$.
\end{lem}
\begin{proof}
Let $(x,y,(xy)^{-1})$ be a spherical $3-$system for $G$, then the
action of $\mathbf{B}:=\mathbf{B}_3=\langle \sigma_1, \sigma_2 \rangle$ is
given by:
\[
    \sigma_1: (x,y,y^{-1}x^{-1}) \rightarrow
    (xyx^{-1},x,y^{-1}x^{-1})=x(y,x,x^{-1}y^{-1})x^{-1} \in O^G_{(y,x,(yx)^{-1})},
\]
and
\[
    \sigma_2: (x,y,y^{-1}x^{-1}) \rightarrow (x,yy^{-1}x^{-1}y^{-1},y) =
    (x,x^{-1}y^{-1},y) \in O^G_{(x,(yx)^{-1},y)}.
\]
\end{proof}

From Lemma~\ref{lem.Br.Inn} and Lemma~\ref{lem.B3.action} we deduce the following
orbit equality.
\begin{cor}\label{cor.B3.Inn}
$O^G_{T^{un}} = O^{\mathbf{B}_3}_{T^{un}}$.
\end{cor}

Denote by $d=d(G;\tau)$ the number of orbits in the set of spherical
$3-$systems of generators for $G$ of unordered type $\tau$, under
the action of $\mathbf{B}_3$. Then by Corollary~\ref{cor.B3.Inn},

\begin{cor}\label{cor.dG}
$d(G;\tau) = \#\{O^G_{T^{un}}: T \in \mathcal{S}(G,\tau)\}.$
\end{cor}

Now, one can bound the number of Hurwitz components using
$d(G;\tau)$.

\begin{cor}\label{cor.hG}
Let $\tau_1$ and $\tau_2$ be two types, then
\[
\frac{d(G;\tau_1)\cdot d(G;\tau_2)}{2|\Out(G)|} \leq h(G;\tau_1,\tau_2) \leq
d(G;\tau_1) \cdot d(G;\tau_2).
\]
\end{cor}

The proof of the Corollary follows from the following Lemma,
which summarizes some well-known facts regarding group actions on finite sets.

\begin{lem}\label{lem.G.action}
Let $H$ and $K$ be groups and let $X$ and $Y$ be finite sets.

If $H$ acts on $X$ and $x \in X$, we denote
$O^H_x:=\{h \cdot x: h \in H\}$ the orbit of $x$ under the action of $H$ and
$\Omega_H = \Omega_H^X$ the set of orbits in $X$.

\begin{enumerate}\renewcommand{\theenumi}{\it \roman{enumi}}
\item If the groups $H$ and $K$ act on a set $X$, then for every $x \in X$,
\[
|O^H_x| \leq |O^{H \times K}_x| \leq |O^H_x|\cdot|K|.
\]

\item If the groups $H$ and $K$ act on a set $X$, and there is some positive constant
$c$ such that $|O^H_x| \leq c|O^K_x|$ for any $x \in X$, then
\[
    |\Omega_H| \geq \frac{1}{c}|\Omega_K|.
\]

\item If the groups $H$ and $K$ act on the sets $X$ and $Y$ respectively, then
\[
    |\Omega_{H \times K}^{X \times Y}| = |\Omega_H^X| \cdot |\Omega_K^Y|.
\]
\end{enumerate}
\end{lem}

\begin{proof}[Proof of Corollary~\ref{cor.hG}]
By Lemma~\ref{lem.G.action},
\[
\Omega_{(B_3 \times B_3)\times \Aut(G)}^{\mathcal{U}(G;\tau_1,\tau_2)} \leq
\Omega_{B_3 \times B_3}^{\mathcal{U}(G;\tau_1,\tau_2)} = d(G;\tau_1) \cdot d(G;\tau_2).
\]
On the other direction, by Lemma~\ref{lem.Br.Inn},
Lemma~\ref{lem.G.action} and Definition \ref{def.Hur.comp}, we have

In case $\tau_1 \ne \tau_2$:
\begin{align*}
&\Omega_{(B_3 \times B_3)\times \Aut(G)}^{\mathcal{U}(G;\tau_1,\tau_2)} =
\Omega_{(B_3 \times B_3)\times \Out(G)}^{\mathcal{U}(G;\tau_1,\tau_2)}
&\geq \frac{\Omega_{B_3 \times B_3}^{\mathcal{U}(G;\tau_1,\tau_2)}}{|\Out(G)|}
= \frac{d(G;\tau_1) \cdot d(G;\tau_2)}{|\Out(G)|}.
\end{align*}

In case $\tau_1 = \tau_2= \tau$:
\begin{align*}
&\Omega_{(B_3 \wr \ZZ/2\ZZ)\times \Aut(G)}^{\mathcal{U}(G;\tau,\tau)} =
\Omega_{(B_3 \wr \ZZ/2\ZZ)\times \Out(G)}^{\mathcal{U}(G;\tau,\tau)}
&\geq \frac{\Omega_{B_3 \wr \ZZ/2\ZZ}^{\mathcal{U}(G;\tau,\tau)}}{|\Out(G)|}
= \frac{d(G;\tau)^2}{2|\Out(G)|}.
\end{align*}
\end{proof}

\subsection{Hurwitz components for $\mathfrak{A}_n$ and $\mathfrak{S}_n$}\label{subsec.Hur.An.Sn}
In this Section we prove Theorems~\ref{thm.moduli.An}(a)
and~\ref{thm.moduli.Sn}(a) regarding alternating and symmetric
groups. 

The results of Liebeck and Shalev, which are stated below, are applicable to any
Fuchsian group $\Ga$, however, we shall use them only for the case
of orbifold surface groups (see
Definition~\ref{defn.orbifold.fuchsian})
\[
    \Ga=\Ga(g' \mid m_1,\dots,m_r)
\]
that satisfy the inequality
\begin{equation}\label{eq.Fuchsian}
    2g'-2 + \sum_{i=1}^r \bigl( 1-\frac{1}{m_i}\bigr) > 0.
\end{equation}

\begin{defin}
Let $C_i = g_i^{\mathfrak{S}_n}$ ($1 \leq i \leq r$) be conjugacy classes in
$\mathfrak{S}_n$, and let $m_i$ be the order of $g_i$. Define $\sgn(C_i)=
\sgn(g_i)$, and write $\mathbf{C} = (C_1,\dots,C_r)$. Define
\[
    \Hom_{\mathbf{C}}(\Gamma,\mathfrak{S}_n) = \{\phi\in \Hom(\Gamma,\mathfrak{S}_n): \phi(x_i) \in C_i
    \text{ for } 1 \leq i \leq r\}.
\]
\end{defin}

\begin{defin}
Conjugacy classes in $\mathfrak{S}_n$ of cycle-shape $(m^k)$, where $n = mk$,
namely, containing $k$ cycles of length $m$ each, are called
\emph{homogeneous}. A conjugacy class having cycle-shape
$(m^k,1^f)$, namely, containing $k$ cycles of length $m$ each and
$f$ fixed points, with $f$ bounded, is called \emph{almost
homogeneous}.
\end{defin}

\begin{theo}~\cite[Theorem 1.9]{LS04}. \label{thm.conj.An}
Let $\Gamma$ be a Fuchsian group, and let $C_i$ ($1\leq i \leq r$)
be conjugacy classes in $\mathfrak{S}_n$ with cycle-shapes
$(m_i^{k_i},1^{f_i})$, where $f_i<f$ for some constant $f$ and
$\prod_{i=1}^r \sgn(C_i) = 1$. Set $\mathbf{C} = (C_1,\dots,C_r)$.
Then the probability that a random homomorphism in
$\Hom_\mathbf{C}(\Gamma, \mathfrak{S}_n)$ has image containing $\mathfrak{A}_n$ tends to
$1$ as $n \rightarrow \infty$.
\end{theo}

Using Theorem~\ref{thm.conj.An}, Liebeck and Shalev deduced the
following Corollary regarding $\mathfrak{S}_n$.

\begin{cor}\label{cor.conj.Sn} \cite[Theorem 1.10]{LS04}.
Let $\Gamma=\Gamma(-|m_1,\dots,m_r)$ be a polygonal group which
satisfies the above inequality~\eqref{eq.Fuchsian}, and assume that
at least two of $m_1,\dots,m_r$ are even. Then $\Gamma$ surjects to
all but finitely many symmetric groups $\mathfrak{S}_n$.
\end{cor}

As consequences of these results we recall the following theorems of
\cite{GP}.

\begin{theo}\label{cor.unmixed.ram.An}
Let $\tau_1 = (m_{1,1},\dots,m_{1,r_1})$ and $\tau_2 =
(m_{2,1},\dots,m_{2,r_2})$ be two sequences of natural numbers such
that $m_{i,k} \geq 2$ and $\sum_{k=1}^{r_i}(1-1/m_{i,k}) > 2$ for
$i=1,2$. Then almost all alternating groups $\mathfrak{A}_n$ admit an unmixed
ramification structure of type $(\tau_1,\tau_2)$.
\end{theo}

\begin{theo}\label{cor.unmixed.ram.Sn}
Let $\tau_1 = (m_{1,1},\dots,m_{1,r_1})$ and $\tau_2 =
(m_{2,1},\dots,m_{2,r_2})$ be two sequences of natural numbers such
that $m_{i,k} \geq 2$, at least two of $(m_{i,1},\dots,m_{i,r_i})$
are even and $\sum_{k=1}^{r_i}(1-1/m_{i,k}) > 2$, for $i=1,2$. Then
almost all symmetric groups $\mathfrak{S}_n$ admit an unmixed ramification
structure of type $(\tau_1,\tau_2)$.
\end{theo}

The proofs of the two theorems are based on the following generalization of the algorithm 
appearing in \cite{GP}.

\begin{alg}\label{alg.conj.class}
Given two sequences of natural numbers 
$\tau_1 = (m_{1,1},\dots,m_{1,r_1})$ and $\tau_2 = (m_{2,1},\dots,m_{2,r_2})$, such
that $m_{i,k} \geq 2$, then one can choose $r_1+r_2$ almost homogeneous conjugacy classes
$C_{m_{1,1}}$, $\ldots ,C_{m_{1,r_1}}$, $C_{m_{2,1}}, \ldots ,C_{m_{2,r_2}}$ in $\mathfrak{S}_n$, of
orders $m_{1,1}, \dots ,m_{1,r_1},m_{2,1}, \dots ,m_{2,r_2}$ respectively, such that they
contain only even permutations, and they all have different numbers
of fixed points.
\end{alg}

We can now prove Theorems~\ref{thm.moduli.An}(a) and~\ref{thm.moduli.Sn}(a).

\begin{proof}[Proof of Theorem~\ref{thm.moduli.An}(a)]
Let $\tau_1 = (m_{1,1},\dots,m_{1,r_1})$ and $\tau_2 =
(m_{2,1},\dots,m_{2,r_2})$ be two sequences of natural numbers such
that $m_{j,l} \geq 2$ and $\sum_{l=1}^{r_j}(1-1/m_{j,l}) > 2$ for
$j=1,2$. Let $k \in \mathbb{N}$ be an arbitrary integer,
and assume that $n$ is large enough. By slightly modifying
Algorithm~\ref{alg.conj.class}, we may actually choose $(r_1+r_2)k$ almost
homogeneous conjugacy classes in $\mathfrak{S}_n$,
\[
\{C_{m_{1,1},i}, \ldots ,C_{m_{1,r_1},i},C_{m_{2,1},i}, \ldots ,C_{m_{2,r_2},i}\}_{i=1}^k,
\]
which contain even permutations, such that every $r_1+r_2$ classes have
orders $m_{1,1}, \dots ,m_{1,r_1},m_{2,1}, \dots ,m_{2,r_2}$ respectively, and all the $(r_1+r_2)k$
conjugacy classes have different numbers of fixed points.

Hence, if $n$ is large enough, there are $(r_1+r_2)k$ different
$\mathfrak{S}_n$-conjugacy classes in $\mathfrak{A}_n$, and moreover, for each $1\leq
i_1,\ldots,i_{r_1},j_1,\ldots,j_{r_2} \leq k$, the conjugacy classes
$(C_{m_{1,1},i_1},\ldots,C_{m_{1,r_1},i_{r_1}})$, 
$(C_{m_{2,1},j_1},\ldots,C_{m_{2,r_2},j_{r_2}})$ 
contain an unmixed ramification structure $(T_1,T_2)$ of type $(\tau_1,\tau_2)$,
by Theorem~\ref{thm.conj.An} (see the proof of Theorem~\ref{cor.unmixed.ram.An}).

From Lemma~\ref{lem.two.Br.action}, since $\mathfrak{S}_n = \Aut(\mathfrak{A}_n)$ (for
$n>6$), we deduce that if $n$ is large enough, then
$h(\mathfrak{A}_n;\tau_1,\tau_2) \geq k^{r_1+r_2}$. Now, $k$ can be arbitrarily
large, therefore,
\[
    h(\mathfrak{A}_n;\tau_1,\tau_2) \stackrel{n \rightarrow \infty}{\longrightarrow}  \infty.
\]

Moreover, as the number of different almost homogeneous conjugacy
classes in $\mathfrak{S}_n$ of some certain order grows linearly in $n$, the
proof actually shows that $h=\Omega(n^{r_1+r_2})$.
\end{proof}

Similarly, we can show that if $\tau_1 = (m_{1,1},\dots,m_{1,r_1})$ and $\tau_2 =
(m_{2,1},\dots,m_{2,r_2})$ are two sequences of natural numbers such
that $m_{i,l} \geq 2$, at least two of $(m_{l,1},\dots,m_{l,r_l})$
are even and $\sum_{l=1}^{r_i}(1-1/m_{i,l}) > 2$, for $i=1,2$, then
\[
    h(\mathfrak{S}_n;\tau_1,\tau_2) \stackrel{n \rightarrow \infty}{\longrightarrow}
    \infty,
\]
and moreover, $h=\Omega(n^{r_1+r_2})$, thus proving
Theorem~\ref{thm.moduli.Sn}(a).

\subsection{Hurwitz components for $\PSL(2,p)$}\label{subsec.Hu.psl}

In this section we prove Theorem~\ref{thm.moduli.PSLp}. The proof is based on well-known
properties of the group $\PSL(2,p^e)$ (see for example \cite{Su}) and on
results of Macbeath~\cite{Ma}.


Let $q = p^e$, where $p$ is an odd prime and $e \geq 1$. Recall
that $\GL(2,q)$ is the group of invertible $2 \times 2$ matrices
over the finite field with $q$ elements, which we denote by $\FQ$,
and $\SL(2,q)$ is the subgroup of $\GL(2,q)$ comprising the matrices
with determinant $1$. Then $\PGL(2,q)$ and $\PSL(2,q)$ are the
quotients of $\GL(2,q)$ and $\SL(2,q)$ by their respective centers.
Moreover we can identify $\PSL(2,q)$ with a normal subgroup of
index $2$ in $\PGL(2,q)$.

Since all non-trivial elements in $\PSL(2,q)$, whose pre-images in $\SL(2,q)$
have the same trace, are conjugate in $\PGL(2,q)$, all of them have
the same order in $\PSL(2,q)$. Therefore, we may denote by
$\mathcal{O}rd(\al)$ the order in $\PSL(2,q)$ of the image of a non-trivial
matrix $A\in \SL(2,q)$ whose trace equals $\al$, and denote, for an
integer $l$,
\[
    \Tr_l = \{\al \in \mathbb{F}_q: \mathcal{O}rd(\al)=l\}.
\]
Note that since $q$ is odd then $\al \in \Tr_l$ if and only if $-\al
\in \Tr_l$.

Now, one can easily compute the size of $\Tr_l$ for any integer $l$.

\begin{lem}\label{lem.PSL.odd.Tr}
Let $p$ be an odd prime and let $q=p^e$. Then in $\PSL(2,q)$,
\begin{enumerate}\renewcommand{\theenumi}{\it \roman{enumi}}
\item $\Tr_p=\{\pm 2\}$ and so $|\Tr_p|=2$.
\item $\Tr_2=\{0\}$ and so $|\Tr_2|=1$.
\item If $r \geq 3$ and $r \mid \frac{q \pm 1}{2}$ then
$|\Tr_r|=\phi(r)$, where $\phi$ is the Euler function.
\item For other values of $r$, $|\Tr_r|=0$.
\end{enumerate}
\end{lem}

This Lemma is immediate, but for the convenience of the
reader we shall present a proof of part {\it(iii)}.

\begin{proof}[Proof of Lemma~\ref{lem.PSL.odd.Tr} (iii)]
Assume that $r \mid \frac{q-1}{2}$ (if $r \mid \frac{q+1}{2}$ then the proof is similar).
Let $\la$ be a primitive root of unity of order $2r$ in
$\mathbb{F}_q$, then there are $2\phi(r)$ diagonal matrices whose images in
$\PSL(2,q)$ have exact order $r$, parametrized by $\{\pm \la^i:
1\leq i \leq 2r, (i,2r)=1\}$, if $r$ is odd, or by $\{\pm \la^i:
1\leq i \leq r, (i,2r)=1\}$, if $r$ is even. Hence, there are
exactly $\phi(r)$ different traces of diagonal matrices of order $r$, 
given as $\{\pm \al_1,\dots,\pm \al_\psi\}$,
where $\psi=\frac{\phi(r)}{2}$.
\end{proof}


In order to estimate the number of Hurwitz components for
$\PSL(2,p)$, we would first like to estimate the number
$d(\PSL(2,p);\tau)$ for certain types $\tau$ (see
Corollaries~\ref{cor.dG} and~\ref{cor.hG}).

Recall that by Corollary~\ref{cor.dG},
\[
   d(\PSL(2,p);\tau) = \#\{O^{\PSL(2,p)}_{T^{un}}: T \in
   \mathcal{S}(\PSL(2,p),\tau)\}.
\]

Thus, $d(\PSL(2,p);\tau) = 2d'(\PSL(2,p),\tau)$, where
\[
   d'(\PSL(2,p);\tau) := \#\{O^{\PGL(2,p)}_{T^{un}}: T \in
   \mathcal{S}(\PSL(2,p),\tau)\}.
\]

One can compute $d'(\PSL(2,p),\tau)$ using the following Lemma, 
which follows from results of Macbeath \cite{Ma}.

\begin{lem} \label{lem.l.m.n}
Let $2 \leq l \leq m \leq n$ and assume that $m>2$ and $n>5$. Then
\begin{multline*}
    d'(\PSL(2,p);(l,m,n))=
    \# \bigl\{ (\pm\al,\pm\be,\pm\ga): \al \in \Tr_l, \be \in \Tr_m, \ga \in
    \Tr_n, \\ \text{ and either }\al^2+\be^2+\ga^2 - \al\be\ga \neq 4 
    \text{ or } \al^2+\be^2+\ga^2 + \al\be\ga \neq 4 \bigr\}.
\end{multline*}
\end{lem}
\begin{proof}
Let $(\al,\be,\ga) \in \mathbb{F}_p^3$ such that 
$\al \in \Tr_l$, $\be \in \Tr_m$ and $\ga \in \Tr_n$.
In this case, also $-\al \in \Tr_l$, $-\be \in \Tr_m$ and $-\ga \in \Tr_n$.

Then, by \cite[Theorem 1]{Ma}, there exist three matrices $A,B,C \in \SL(2,p)$
such that $ABC=1$, $\tr(A)=\al$, $\tr(B)=\be$ and $\tr(C)=\ga$. 
Denote by $\bar{A}$ the image of a matrix $A$ in $\PSL(2,p)$.
Then, $\bar{A}$ has order $l$, $\bar{B}$ has order $m$ and $\bar{C}$ has order $n$.

By \cite[Theorem 4]{Ma}, the group generated by $\bar{A}$ and $\bar{B}$ is $\PSL(2,p)$
if and only if either $\al^2+\be^2+\ga^2 - \al\be\ga \neq 4$ or
$\al^2+\be^2+\ga^2 + \al\be\ga \neq 4$.

In addition, if there exist some other matrices $A`,B`,C` \in \SL(2,p)$
such that $A`B`C`=1$, $\tr(A`)=\pm \al$, $\tr(B`)=\pm \be$ and $\tr(C`)=\pm \ga$
then by \cite[Theorem 3]{Ma}, there is some $g \in \PGL(2,p)$ such that
$g \bar{A} g^{-1} = \bar{A`}$ and $g \bar{B} g^{-1} = \bar{B`}$ 
implying that also $g \bar{C} g^{-1} = g \bar{B}^{-1}\bar{A}^{-1} g^{-1} = \bar{C`}$.
\end{proof}

\begin{cor} \label{cor.d.PSL}
Let $p \geq 5$ be an odd prime, then in $\PSL(2,p)$,
\begin{enumerate}\renewcommand{\theenumi}{\it \roman{enumi}}

\item $d'(\PSL(2,p);(2,3,p)) = 1$.

\item If $r \geq 7$ and $r \mid
\frac{p \pm 1}{2}$ then $d'(\PSL(2,p);(2,3,r)) =
\frac{\phi(r)}{2}$.

\item $d'(\PSL(2,p);(p,p,p)) = 1$.

\item If $r
\geq 7$ and $r \mid \frac{p \pm 1}{2}$ then
\[
d'(\PSL(2,p);(r,r,r)) = \frac{\psi(\psi+1)(\psi+2)}{6},
\]
where $\psi=\frac{\phi(r)}{2}$.

\item If $2 < l < m < n$ such that $n>5$ and $l,m,n$ all divide $\frac{p \pm
1}{2}$, then
\[
d'(\PSL(2,p);(l,m,n)) = \frac{\phi(l)\phi(m)\phi(n)}{8}.
\]

\item If $2 \leq l \leq m \leq n$ such that $m>2$ and $n>5$ then
\[
d'(\PSL(2,p);(l,m,n)) \leq \frac{\phi(l)\phi(m)\phi(n)}{8}.
\]
\end{enumerate}
\end{cor}

\begin{proof} 
The proof is based on Lemma~\ref{lem.PSL.odd.Tr} and
Lemma~\ref{lem.l.m.n}.

\begin{enumerate}\renewcommand{\theenumi}{\it \roman{enumi}}
\item The orders $(2,3,p)$ correspond to the traces $(0,\pm 1, \pm
2)$.

\item The orders $(2,3,r)$ correspond to the traces $(0,\pm 1, \pm
\ga)$, with $\mathcal{O}rd(\ga)=r$. We need to verify that $0^2 + 1^2 + \ga^2 - 0 \ne 4$. 
Indeed, $0^2 + 1^2 + \ga^2 - 0 = 4$ is equivalent to $\ga^2=3$, and $\ga^2=3$ if and only if
$\mathcal{O}rd(\ga)=6$, a contradiction.

\item The orders $(p,p,p)$ correspond to the traces $(-2,-2,2)$
(see \cite[Theorem 7]{Ma}).

\item The orders $(r,r,r)$ correspond to the traces
$(\pm \al_i,\pm \al_j,\pm \al_k)$ for $1 \leq i \leq j \leq k \leq
\psi$. If $\al_i^2 + \al_j^2 + \al_k^2 - \al_i\al_j\al_k=4$, then
$\al_i^2 + \al_j^2 + \al_k^2 - \al_i\al_j\al_k \neq 4$, hence, if
necessary, we may replace $(\al_i,\al_j,\al_k)$ by
$(-\al_i,-\al_j,-\al_k)$. Therefore,
\[
d'(\PSL(2,p);(r,r,r)) = {\psi \choose 3} + 2{\psi \choose 2} + \psi
= \frac{\psi(\psi+1)(\psi+2)}{6}.
\]

\item The orders $(l,m,n)$ correspond to the traces
$(\al,\be,\ga)$ where $\mathcal{O}rd(\al)=l$,
$\mathcal{O}rd(\be)=m$, $\mathcal{O}rd(\ga)=n$, and $\al,\be,\ga
\neq 0$. Now, we may replace $(\al,\be,\ga)$ by $(-\al,-\be,-\ga)$,
if necessary.

\item This follows from the previous calculations.
\end{enumerate}
\end{proof}

\begin{proof}[Proof of Theorem~\ref{thm.moduli.PSLp}]
Let $p$ be an odd prime, and let $\tau_1=(r_1,s_1,t_1)$ and
$\tau_2=(r_2,s_2,t_2)$ be two hyperbolic types. By
Corollary~\ref{cor.d.PSL}, for $i=1,2$, $d'(\PSL(2,p);(r_i,s_i,t_i))$
is maximal when $r_i,s_i$ and $t_i$ are three different integers
dividing $\frac{p\pm 1}{2}$, and hence is at most
$\frac{\phi(r_i)\phi(s_i)\phi(t_i)}{8}$.

Define the following constant
\[
c :=
\frac{\phi(r_1)\phi(s_1)\phi(t_1)\phi(r_2)\phi(s_2)\phi(t_2)}{16}.
\]
Then, by Corollary~\ref{cor.hG},
\[
h(G;\tau_1,\tau_2) \leq d(G;\tau_1) \cdot d(G;\tau_2) =
2d'(G;\tau_1) \cdot 2d'(G;\tau_2) \leq c.
\]
\end{proof}

\subsection{Ramification structures and Hurwitz components for abelian groups and their extensions}\label{subsec.hur.abelian}
In this Section we generalize previous results regarding abelian
groups and their extensions, which appeared in~\cite{BCG05}.

\subsubsection{Ramification structures of abelian groups}
The following Theorem generalizes \cite[Theorem 3.4]{BCG05} in case
$G$ abelian and $S$ is isogeneous to a higher product (not
necessarily Beauville).

From now on we use the additive notation for abelian groups.

\begin{theo}\label{thm.unmixed.abelian}
Let $G$ be an abelian group, given as
\[
    G \cong \mathbb{Z}/{n_1}\mathbb{Z} \times \dots \times
    \mathbb{Z}/{n_t}\mathbb{Z},
\]
where $n_1 \mid \dots \mid n_t$. For a prime $p$, denote by $l_i(p)$
the largest power of $p$ which divides $n_i$ (for $1\leq i \leq t$).

Let $r_1,r_2 \geq 3$, then $G$ admits an unmixed ramification
structure of size $(r_1,r_2)$ if and only if the following
conditions hold:

\begin{itemize}
\item $r_1,r_2 \geq t+1$;
\item $n_t=n_{t-1}$;
\item If $l_{t-1}(3)>l_{t-2}(3)$ then $r_1,r_2\geq 4$;
\item $l_{t-1}(2)=l_{t-2}(2)$;
\item If $l_{t-2}(2)>l_{t-3}(2)$ then $r_1,r_2\geq 5$ and $r_1,r_2$
are not both odd.
\end{itemize}

\end{theo}

\begin{proof}
Let $(x_1,\dots,x_{r_1};y_1,\dots,y_{r_2})$ be an unmixed
ramification structure of size $(r_1,r_2)$. Set
\[
    \Sigma_1:=\Sigma(x_1,\dots,x_{r_1}):=\{i_1
    x_1,\dots,i_{r_1}x_{r_1}: i_1,\dots i_{r_1} \in \mathbb{Z}\},
\]
and
\[
    \Sigma_2:=\Sigma(y_1,\dots,y_{r_2}):=\{j_1
    y_1,\dots,j_{r_2}y_{r_2}: j_1,\dots j_{r_2} \in \mathbb{Z}\},
\]
and recall that $\Sigma_1 \cap \Sigma_2 = \{0\}$.

Consider the primary decomposition of $G$,
\[
G = \bigoplus_{p\in \{\text{Primes}\}} G_p,
\]
and observe that since $G$ is generated by $\min\{r_1,r_2\}-1$
elements, so is any $G_p$ (which is a characteristic subgroup of
$G$).

Therefore, $G_p$ can be written as
\[
    G_p \cong \mathbb{Z}/p^{k_1}\mathbb{Z} \times \dots \times \mathbb{Z}/p^{k_{t-1}}\mathbb{Z}
    \times \mathbb{Z}/p^{k_t}\mathbb{Z},
\]
where $k_1 \leq \dots \leq k_{t-1} \leq k_t$ and $1\leq t \leq
\min\{r_1,r_2\}-1$.

Denote $H_p := p^{k_t-1}G_p$, and observe that $H_p$ is an
elementary abelian group of rank at most $t$.

{\bf Step 1.} Let $x_1=(x_{1,p}) \in \bigoplus_{p\in
\{\text{Primes}\}} G_p$ and let
\[
\Sigma_{1,p}:=\Sigma(x_{1,p},\dots,x_{r_1,p}):= \{l_1
    x_{1,p},\dots,l_{r_1}x_{r_1,p}: l_1,\dots l_{r_1} \in \mathbb{Z}\},
\]
be the set of multiples of $(x_{1,p},\dots,x_{r_1,p})$, then by the
Chinese Remainder Theorem, $x_{1,p}$ is a multiple of $x_1$, and
hence $\Sigma_1 \supseteq \Sigma_{1,p}$.

{\bf Step 2.} $G_p$ is not cyclic.

Otherwise, if $G_p \cong \mathbb{Z}/p^k\mathbb{Z}$, then $H_p =
p^{k-1}G_p \cong \mathbb{Z}/p\mathbb{Z}$. Since $\Sigma_{1,p}$
contains a generator of $G_p$, it also contains a non-trivial
element of $H_p$ and so $\Sigma_{1,p} \supseteq H_p$. Thus $\Sigma_1
\supseteq H_p$, and similarly $\Sigma_2 \supseteq H_p$, a
contradiction to $\Sigma_1 \cap \Sigma_2 = \{0\}$.

{\bf Step 3.} $k_t = k_{t-1}$, namely $G_p \cong
\mathbb{Z}/p^{k_1}\mathbb{Z} \times \dots \times
\mathbb{Z}/p^{k_{t-1}}\mathbb{Z} \times
\mathbb{Z}/p^{k_{t-1}}\mathbb{Z}$, where $k_1 \leq \dots \leq
k_{t-1}$ and $2\leq t \leq \min\{r_1,r_2\}-1$.

Otherwise, if $k_t \neq k_{t-1}$, then $H_p = p^{k_t-1}G_p \cong
\mathbb{Z}/p\mathbb{Z}$. As in Step 2, $\Sigma_{1,p}$ contains a
generator of $G_p$, and so it also contains a non-trivial element of
$H_p$. Thus $\Sigma_{1,p} \supseteq H_p$, and similarly
$\Sigma_{2,p} \supseteq H_p$, a contradiction to $\Sigma_1 \cap
\Sigma_2 = \{0\}$.

{\bf Step 4.} $p=2$ or $3$.

The extra conditions for $p=2$ and $3$ are due to dimensional
reasons.

\begin{itemize}
\item Let $p=2$ and assume that $k_{t-1}>k_{t-2}$.
In this case, $H_2 \cong (\mathbb{Z}/2\mathbb{Z})^2$ contains only
three non-trivial vectors. However, $|H_2 \cap \Sigma_{1,2}| \geq 2$
and $|H_2 \cap \Sigma_{2,2}| \geq 2$, a contradiction to $\Sigma_1
\cap \Sigma_2 = \{0\}$.

\item Let $p=2$ and assume that $k_{t-1}=k_{t-2}>k_{t-3}$. In this case,
$H_2 \cong (\mathbb{Z}/2\mathbb{Z})^3$ contains only seven
non-trivial vectors.

If $r_1=4$ then $\Sigma_{1,2}$ contains four different vectors
which generate $H_2$, whose sum is zero, say
$\{e_1,e_2,e_3,e_1+e_2+e_3\}$. Now, the other three vectors in $H_2$
are necessarily $\{e_1+e_2,e_1+e_3,e_2+e_3\}$, which are linearly
dependent, and so cannot generate $H_2 \cong
(\mathbb{Z}/2\mathbb{Z})^3$.

When $r_1$ is odd, $\Sigma_{1,2}$ contains four different vectors
from $H_2$. Indeed, a sum $x_1+\dots+x_{r_1}$ of some vectors
$v,u,w$ over $\mathbb{Z}/2\mathbb{Z}$ (i.e. $x_i \in \{v,u,w\}$),
where $r_1$ is odd, cannot be equal to $0$, unless $v$, $u$ and $w$
are linearly dependent,  and so cannot generate $H_2 \cong
(\mathbb{Z}/2\mathbb{Z})^3$.
Thus, if $r_1$ is odd, then $|H_2 \cap
\Sigma_{1,2}| \geq 4$, and similarly, if $r_2$ is odd, then $|H_2
\cap \Sigma_{2,2}| \geq 4$, a contradiction to $\Sigma_1 \cap
\Sigma_2 = \{0\}$.

\item Let $p=3$ and assume that $k_{t-1}>k_{t-2}$.
In this case, $H_3 \cong (\mathbb{Z}/3\mathbb{Z})^2$ contains only
eight non-trivial vectors. If $r_1=3$ then $\Sigma_{1,3}$ contains
three different vectors, which generate $H_3$, whose sum is zero, say
$\{e_1,e_2,2e_1+2e_2\}$, as well as their multiples $\{2e_1,2e_2,e_1+e_2\}$.
Now, the other two vectors in $H_2$ are necessarily $\{e_1+2e_2,2e_1+e_2\}$,
which are linearly dependent, and so cannot generate $H_3 \cong
(\mathbb{Z}/3\mathbb{Z})^2$.
\end{itemize}

{\bf Step 5.} Now, let $p \geq 5$ and assume that $G_p =
\mathbb{Z}/p^{k_1}\mathbb{Z} \times \dots \times
\mathbb{Z}/p^{k_{t-1}}\mathbb{Z} \times
\mathbb{Z}/p^{k_{t-1}}\mathbb{Z}$, where $k_1 \leq \dots \leq
k_{t-1}$ and $2\leq t \leq \min\{r_1,r_2\}-1$. We will choose
appropriate vectors for $\Sigma_{1,p}$ and $\Sigma_{2,p}$.

Assume that $(a,b,c,d)$ satisfy the condition in
Equation~\eqref{eq.sol.Zp} below, and let
\begin{align*}
x_{1,p} &= (1,0,\dots,0,1,0) & y_{1,p} &= (1,0,\dots,0,a,b) \\
x_{2,p} &= (0,1,0,\dots,0,0,1) & y_{2,p} &= (0,1,0,\dots,0,c,d) \\
x_{3,p} &= (0,0,1,0,\dots,0,-1,0) & y_{3,p} &= (0,0,1,0,\dots,0,-a,-b) \\
x_{4,p} &= (0,0,0,1,0\dots,0,0,-1) & y_{4,p} &= (0,0,0,1,0,\dots,0,-c,-d) \\
&\vdots & &\vdots \\
x_{t-2,p} &= (0,\dots,0,1,*,*) & y_{t-2,p} &= (0,\dots,0,1,*,*) \\
x_{t-1,p} &= (0,\dots,0,0,*,*) & y_{t-1,p} &= (0,\dots,0,0,*,*) \\
x_{t,p} &= (0,\dots,0,0,*,*) & y_{t,p} &= (0,\dots,0,0,*,*) \\
&\vdots & &\vdots\\
x_{r_1,p} &= (-1,\dots,-1,-1,-1) & y_{r_2,p} &= (-1,\dots,-1,-a-c,-b-d) \\
\end{align*}
where the elements marked with $(*,*)$ in $x_{t-2,p}$ (and after)
are chosen from $\{(0,\pm 1),(\pm 1,0), \pm(1,1)\}$ such that
$(x_{1,p},x_{2,p},\dots,x_{t,p})$ are independent and the sum
$x_{1,p}+\dots+x_{r_1,p}=0$. Similarly, the elements marked with
$(*,*)$ in $y_{t-2,p}$ (and after) are chosen from
$\{\pm(a,b),\pm(c,d), \pm(a+c,b+d)\}$, such that
$(y_{1,p},y_{2,p},\dots,y_{t,p})$ are independent and
$y_{1,p}+\dots+y_{r_1,p}=0$.

Since $\langle x_{1,p},\dots,x_{r_1,p} \rangle = G_p = \langle
y_{1,p},\dots,y_{r_2,p} \rangle$, we deduce that
$(x_{1,p},\dots,x_{r_1,p})$ form a spherical $r_1-$system of
generators for $G_p$ and that $(y_{1,p},\dots,y_{r_2,p})$ form a
spherical $r_2-$system of generators for $G_p$. Moreover, for every
$1 \leq i \leq r_1$, $1 \leq j \leq r_2$, and $k,l \in \mathbb{Z}$,
if the vectors $kx_{i,p}$ and $ly_{j,p}$ are not trivial, then they
are linearity independent. Hence, $\Sigma_{1,p} \cap \Sigma_{2,p} =
\{0\}$, as needed.

When $p=2$ or $3$ it suffices to construct unmixed ramification
structures for the elementary abelian groups in characteristic $2$
and $3$. These yield an unmixed ramification structure for any
choice of $H_2$ (resp. $H_3$), which induces an appropriate
structure for any $G_2$ (resp. $G_3$), by completing the generating
vectors of $H_2$ (resp. $H_3$) to generating vectors of $G_2$ (resp.
$G_3$), essentially in the same way of $p \geq 5$. These
constructions are described in the following
Lemmas~\ref{lem.unmixed.2} and~\ref{lem.unmixed.3}.

Now, recall that by using the primary decomposition of $G$, it was
enough to check the conditions on each primary component $G_p$, thus
$G$ admits an unmixed ramification structure of size $(r_1,r_2)$ as
needed.
\end{proof}

\begin{lem}\label{lem.unmixed.2}
Let $G=(\mathbb{Z}/2\mathbb{Z})^t$.

If $t \geq 4$ then $G$ always admits an unmixed ramification
structure of size $(r_1,r_2)$, for any $r_1,r_2 \geq t+1$.

If $t=3$ then $G$ admits an unmixed ramification structure of size
$(r_1,r_2)$, if and only if $r_1,r_2 \geq 5$ and $r_1$, $r_2$ are
not both odd.
\end{lem}

\begin{lem}\label{lem.unmixed.3}
Let $G=(\mathbb{Z}/3\mathbb{Z})^t$.

If $t \geq 3$ then $G$ always admits an unmixed ramification
structure of size $(r_1,r_2)$, for any $r_1,r_2 \geq t+1$.

If $t=2$ then $G$ admits an unmixed ramification structure of size
$(r_1,r_2)$, if and only if $r_1,r_2 \geq 4$.
\end{lem}

The proofs of both lemmas are straight forward calculation.

\begin{lem}\label{lem.sol.p}
Let $p \geq 5$ be a prime number and
$U:=(\mathbb{Z}/p\mathbb{Z})^*$, the number $N$ of quadruples
$(a,b,c,d) \in U$  such that:
\begin{equation}\label{eq.sol.Zp}
    a-b, a+c, c-d, b+d, a+c-b-d, ad-bc \in U
\end{equation}
is $N=(p-1)(p-2)(p-3)(p-4)$.
\end{lem}
\begin{proof}
The number $N$ equals $p-1$ times the number of solutions that we
get for $a=1$. Now, $b \neq 0,1$, so there are $p-2$ possibilities
for $b$. The conditions $c \neq 0, -1$ and $d \neq 0, -b$ imply
$(p-2)^2$ possibilities for the pair $(c,d)$. From this number we
need to subtract the number of solutions for $c=d$, $d=1-b+c$ and
$d=bc$, which are $p-2$, $p-2$ and $p-4$ respectively. We deduce
that there are $(p-2)^2 - [(p-2)+(p-2)+(p-4)]=(p-3)(p-4)$
possibilities for the pair $(c,d)$. Hence $N=(p-1)(p-2)(p-3)(p-4)$.
\end{proof}

We remark that this Lemma corrects the calculation given
in~\cite[Theorem 3.4]{BCG05}.

\subsubsection{Hurwitz components in case $(\mathbb{Z}/n\mathbb{Z})^2$}
Observe that for $G=(\mathbb{Z}/n\mathbb{Z})^2$ there is only one
type of a spherical $3-$system of generators, which is
$\tau=(n,n,n)$. Also note that $\Aut(G) \cong \GL(2,n)$.

The following Lemmas give a more precise estimation of the number
of Hurwitz components in case $G=(\mathbb{Z}/n\mathbb{Z})^2$,
which generalizes Remark 3.5 in~\cite{BCG05}.

\begin{lem}\label{lem.ZnZ.p}
Let $p\geq 5$ be a prime. The number $h=h(G;\tau,\tau)$, where
$\tau=(p,p,p)$, of Hurwitz components for
$G=(\mathbb{Z}/p\mathbb{Z})^2$ satisfies
\[
N_p/72 \leq h \leq N_p/6,
\]
where $N_p=(p-1)(p-2)(p-3)(p-4)$.
\end{lem}

\begin{proof}
Let $(x_1,x_2;y_1,y_2)$ be an unmixed Beauville structure for $G$.
Since $x_1, x_2$ are generators of $G$, they are a basis, and
without loss of generality $x_1, x_2$ are the standard basis
$x_1=(1,0)$, $x_2=(0,1)$. Now, let $y_1=(a,b)$, $y_2=(c,d)$, then
the condition $\Sigma_1 \cap \Sigma_2 = \{0\}$ means that any pair
of the six vectors yield a basis of $G$, implying that $a,b,c,d$
must satisfy the conditions given in Equation~\eqref{eq.sol.Zp}.

Moreover, the $N_p$ pairs $\bigl((1,0),(0,1);(a,b),(c,d)\bigr)$,
where $a,b,c,d$ satisfy~\eqref{eq.sol.Zp}, are exactly the
representatives for the $\Aut(G)-$orbits in the set
$\mathcal{U}(G;\tau,\tau)$.

Now, one should consider the action of $B_3 \times B_3$ on
$\mathcal{U}(G;\tau,\tau)$, which is equivalent to the action of
$\mathfrak{S}_3 \times \mathfrak{S}_3$, since $G$ is abelian. The action of $\mathfrak{S}_3$ on the
second component is obvious (there are $6$ permutations), and the
action of $\mathfrak{S}_3$ on the first component can be translated to an
equivalent $\Aut(G)-$action, given by multiplication in one of the
six matrices:
\[
\left(%
\begin{array}{cc}
  1 & 0 \\
  0 & 1 \\
\end{array}%
\right),
\left(%
\begin{array}{cc}
  0 & 1 \\
  1 & 0 \\
\end{array}%
\right),
\left(%
\begin{array}{cc}
  -1 & 0 \\
  -1 & 1 \\
\end{array}%
\right),
\left(%
\begin{array}{cc}
  1 & -1 \\
  0 & -1 \\
\end{array}%
\right),
\left(%
\begin{array}{cc}
  -1 & 1 \\
  -1 & 0 \\
\end{array}%
\right),
\left(%
\begin{array}{cc}
  0 & -1 \\
  1 & -1 \\
\end{array}%
\right),
\]
yielding an equivalent representative.

Therefore, the action of $\mathfrak{S}_3$ on the second component yields
orbits of length $6$, and the action of $\mathfrak{S}_3$ on the first
component connects them together, and gives orbits of sizes from
$6$ to $36$. Moreover since one can exchange the vector $(x_1,x_2)$ with the vector $(y_1,y_2)$ we get the desired result.
\end{proof}


\begin{cor}\label{cor.ZnZ.pe}
Let $p\geq 5$ be a prime. The number $h=h(G;\tau,\tau)$, where
$\tau=(p^k,p^k,p^k)$, of Hurwitz components for
$G=(\mathbb{Z}/p^k\mathbb{Z})^2$ satisfies
\[
N_{p^k}/72 \leq h \leq N_{p^k}/6,
\]
where $N_{p^k} = p^{4k-4}(p-1)(p-2)(p-3)(p-4)$.
\end{cor}

\begin{proof}
In this case, the number $N_{p^k}$ of $\Aut(G)-$orbits in the set
$\mathcal{U}(G;\tau,\tau)$ is exactly $p^{4k-4}$ times $N_p$, and
the proof is the same as in the previous Lemma~\ref{lem.ZnZ.p}.
\end{proof}

\begin{cor}\label{cor.ZnZ.n}
Let $n$ be an integer s.t. $(n,6)=1$. The number $h=h(G;\tau,\tau)$,
where $\tau=(n,n,n)$, of Hurwitz components for
$G=(\mathbb{Z}/n\mathbb{Z})^2$, where $n=p_1^{k_1}\cdot\ldots\cdot
p_t^{k_t}$, satisfies
\[
N_n/72 \leq h \leq N_n/6,
\]
where $N_n=\prod_{i=1}^t p_i^{4k_i-4}(p_i-1)(p_i-2)(p_i-3)(p_i-4)$.
\end{cor}

\begin{proof}
By the Chinese Remainder Theorem, the number $N_n$ of
$\Aut(G)-$orbits in the set $\mathcal{U}(G;\tau,\tau)$ can be
computed using Corollary~\ref{cor.ZnZ.pe}, and the proof is now
the same as in Lemma~\ref{lem.ZnZ.p}.
\end{proof}

\begin{rem} It clearly follows that
$N_n=O(n^4)$. In addition, in \cite{GJT} is given an explicit
formula for $N_n$.

Notice that if $n$ is divisible by the first $l$ primes $p_i\geq5$ then since:
\[ \lim_{l \rightarrow \infty} \prod_i(1 - \frac{1}{p_i})=0
\]
we have $N_n/n^4 \rightarrow 0$ as $l \rightarrow \infty$. 
\end{rem}

\subsubsection{Hurwitz components in case $G=(\ZZ/p\ZZ)^r$}
Fix an integer $r$, let $p>5$ be a prime number, and let
$G=(\ZZ/p\ZZ)^r$, then by Theorem~\ref{thm.unmixed.abelian}, $G$
admits an unmixed ramification structure of type $(\tau_1,\tau_2)$
where $\tau_1=\tau_2=\tau=(p,\dots,p)$ ($p$ appears $(r+1)-$times)
and $r_1=r_2=r+1$.

\begin{prop}\label{prop.ZpZ.mult}
Fix an integer $r \geq 2$, then the number $h=h(G;\tau,\tau)$ of
Hurwitz components for $G=(\ZZ/p\ZZ)^r$ and $\tau=(p,\dots,p)$ ($p$
appears $(r+1)-$times) satisfies, as $p \rightarrow \infty$,
\[
    h=\Theta(p^{r^2}).
\]
\end{prop}

\begin{proof}
Let $(x_1,\dots,x_{r+1};y_1,\dots,y_{r+1})$ be an unmixed
ramification structure for $G$. Since $x_1,\dots,x_{r+1}$ generate
$G$, they are a basis, and without loss of generality they are of
the form given in Step 5 of Theorem~\ref{thm.unmixed.abelian}.
However, for $y_1,\dots,y_{r+1}$ one can take any appropriate set of
$r+1$ vectors in $(\ZZ/p\ZZ)^r$, which admit an unmixed ramification
structure, and so each proper choice of $(y_1,\dots,y_{r+1})$
corresponds to exactly one $\Aut(G)-$orbit in the set
$\mathcal{U}(G;\tau,\tau)$.

Therefore, one can choose any invertible $(r-2)\times (r-2)$ matrix
for
\[
\begin{pmatrix}
y_{1,1}& \dots & y_{1,r-2}\\
& \vdots & \\
y_{r-2,1}& \dots & y_{r-2,r-2}\\
\end{pmatrix},
\]
choose any vector of length $r-2$ for
$(y_{r-1,1},\dots,y_{r-1,r-2})$, and similarly for
$(y_{r,1},\dots,y_{r,r-2})$. Moreover, for $1 \leq i \leq r-2$, one
can choose for $(y_{i,r-1},y_{i,r})$ any vector from the set
$S:=\{(a,b)\in \mathbb{F}_p^2: a \neq 0, b \neq 0, a \neq b\}$.
Observe that $|S|=(p-1)(p-2)$.

Now, one has to make sure that $y_{r-1}$ is not a linear combination
of $y_1,\dots,y_{r-2}$, by choosing $(y_{r-1,r-1},y_{r-1,r})$
appropriately from $S$, and so there are at least
$(p-1)(p-2)-1=p^2-3p+1$ possibilities for this pair. Moreover, one
should choose $(y_{r,r-1},y_{r,r})$ appropriately from $S$, such
that $y_{r}$ is not some linear combination of $y_1,\dots,y_{r-1}$,
and that $(y_{r+1,r-1},y_{r+1,r}) \in S$, and so the number of
possibilities to the pair $(y_{r,r-1},y_{r,r})$ is at least
$(p-3)(p-5)=p^2-8p+15$.

The condition that the pairs $(y_{i,r-1},y_{i,r}) \in S$ for $1 \leq
i \leq r+1$ is needed to guarantee that for any $k,l \in \mathbb{Z}$
and $1\leq i,j\leq r+1$, if the vectors $kx_i$ and $ly_j$ are not
trivial, then they are linearity independent, and so $\Sigma_1 \cap
\Sigma_2 = \{0\}$, as needed.

Hence, the number of $\Aut(G)-$orbits in the set
$\mathcal{U}(G;\tau,\tau)$ is bounded from below by
\begin{align*}
    &|\GL((r-2),p)|p^{2(r-2)}\bigl((p-1)(p-2)\bigr)^{r-2}(p^2-3p+1)(p^2-8p+15) \\
    &= \Theta\bigl(p^{(r-2)^2+2(r-2)+2(r-2)+2+2}\bigr) =
    \Theta\bigl(p^{r^2}\bigr).
\end{align*}
It is clear that the number of orbits is bounded from above by
\[
    |\GL(r,p)|=\Theta(p^{r^2}).
\]

Now, the action of $B_{r+1} \times B_{r+1}$ on the $\Aut(G)-$orbits
of $\mathcal{U}(G;\tau,\tau)$, is equivalent to the action of
$\mathfrak{S}_{r+1} \times \mathfrak{S}_{r+1}$, since $G$ is abelian, and so yields orbits
of sizes between $(r+1)!$ and $((r+1)!)^2$. This has no effect on
the above asymptotic, however, since $r$ is fixed.
\end{proof}

\subsubsection{Extensions of abelian groups, dihedral groups and small groups}\label{sect.extension}
\begin{proof}[Group theoretical proof to Lemma~\ref{lem.genus.ge.2}]

If $r_1=r_2=3$, then the result follows from~\cite[Proposition
3.2]{BCG05}. Note that when $r_i=3$ (for $i=1$ or $2$) the
above inequality is equivalent to the condition that $\mu_i<1$.

If $\mu>1$ then the possible unordered types are
\[
    (2,2,n) (n \in \mathbb{N}),\quad (2,3,3),\quad (2,3,4),\quad (2,3,5).
\]

In the first case, $G\cong D_n$ is a dihedral group of order $2n$,
and thus cannot admit an unmixed Beauville structure by~\cite[Lemma
3.7]{BCG05}. Moreover, $G$ cannot admit an unmixed ramification
structure $(T_1,T_2)$, where $T_1$ has an unordered type $(2,2,n)$.
Indeed, let $C_n$ denote a maximal cyclic subgroup of $D_n$, then
$D_n \setminus C_n$ contains at most two conjugacy classes, more
precisely, it contains one if $n$ is odd and two if $n$ is even. If
$n$ is odd, since both $T_1$ and $T_2$ contain elements of $D_n
\setminus C_n$, then $\Sigma_1 \cap \Sigma_2 \neq \{1\}$. If $n$ is
even, then $T_1$ necessarily contains two elements from two
different conjugacy classes of $D_n \setminus C_n$, and $T_2$ always
contains an element of $D_n \setminus C_n$, which again contradicts
$\Sigma_1 \cap \Sigma_2 = \{1\}$.

In the other cases, one obtains the following isomorphisms of
triangular groups
\[
    \Delta(2,3,3)\cong \mathfrak{A}_4,\quad \Delta(2,3,4)\cong \mathfrak{S}_4,\quad \Delta(2,3,5)\cong \mathfrak{A}_5,
\]
and it is easy to check that these groups do not admit an unmixed
Beauville structure (see also~\cite[Proposition 3.6]{BCG05}).
Moreover, these groups cannot admit an unmixed ramification
structure $(T_1,T_2)$, where $T_1$ has an unordered type $(2,3,n)$
and $n=3,4,5$. Indeed, in the groups $\mathfrak{A}_4$ and $\mathfrak{A}_5$, any two
elements of the same order are either conjugate or one can be
conjugated to some power of the other, in contradiction to $\Sigma_1
\cap \Sigma_2 = \{1\}$. For the group $\mathfrak{S}_4$, $T_1$ necessarily
contains one $2$-cycle, one $3$-cycle and one $4$-cycle, so the
condition $\Sigma_1 \cap \Sigma_2 = \{1\}$ implies that $T_2$ can
contain only elements which have exactly two $2$-cycles, and these
elements cannot generate $\mathfrak{S}_4$, yielding a contradiction.

If $\mu=1$ then the possible unordered types are
\[
    (3,3,3),\quad (2,4,4), \quad (2,3,6),
\]
and so $G$ is a finite quotient of one of the wall-paper groups and
cannot admit an unmixed Beauville structure by~\cite[\S6]{BCG05}.
Moreover, the arguments in~\cite[\S6]{BCG05} show that, in fact,
these groups cannot admit an unmixed ramification structure
$(T_1,T_2)$, where $T_1$ has an unordered type either $(3,3,3)$ or
$(2,4,4)$ or $(2,3,6)$. For example, if $G$ is a quotient of the
triangle group $\Delta(3,3,3)$, and we denote by $A$ the maximal
normal abelian subgroup of $G$, then by~\cite[Proposition
6.3]{BCG05}, for any $g \in G \setminus A$ there exists some integer
$i$ s.t. $g^i$ belongs to one of two fixed conjugacy classes $C_1$
and $C_2$. Moreover, $T_1$ necessarily contains two elements
$g_1,g_2 \in G \setminus A$ such that $g_1^{i_1} \in C_1$ and
$g_2^{i_2} \in C_2$ for some $i_1,i_2$. Since $T_2$ always contains
an element of $G \setminus A$, this contradicts $\Sigma_1 \cap
\Sigma_2 = \{1\}$.

For $r_1,r_2 \geq 4$ the above inequality holds, unless the type is
$(2,2,2,2)$. In the latter case, $G$ is a finite quotient of the
wall-paper group
\begin{align*}
    \Ga &\cong \langle t_1,t_2,t_3,t_4:
    t_1^2, t_2^2, t_3^2, t_4^2, t_1t_2t_3t_4 \rangle
    \cong \langle t_1,t_2,t_3:
    t_1^2, t_2^2, t_3^2, (t_1t_2t_3)^2 \rangle \\
    &\cong \langle t,r,s: [r,s],t^2,trtr,tsts \rangle,
\end{align*}
by setting $t_1=rt$, $t_2=ts$ and $t_3=t$.

Hence, $\Ga \cong \mathbb{Z}/2\mathbb{Z} \ltimes \mathbb{Z}^2$, and
so all its finite quotients are of the form $G =
\mathbb{Z}/2\mathbb{Z} \ltimes (\mathbb{Z}/m\mathbb{Z} \times
\mathbb{Z}/n\mathbb{Z})$ for some $m,n \in \mathbb{N}$. We will show
in Proposition~\ref{prop.semi.abelian} that these groups cannot
admit an unmixed ramification structure of size $(4,4)$. In fact,
the same argument also shows that these groups cannot admit an
unmixed ramification structure $(T_1,T_2)$, where $T_1$ has an
unordered type $(2,2,2,2)$ (see Remark~\ref{rem.semi.abelian}).
\end{proof}

The above proof uses the following proposition, which generalizes the result in~\cite[Lemma
3.7]{BCG05}, that dihedral groups do not admit unmixed Beauville
structures.

\begin{prop}\label{prop.semi.abelian}
For any $n,m \in \mathbb{N}$, the finite group
\[
G= \mathbb{Z}/2\mathbb{Z} \ltimes (\mathbb{Z}/m\mathbb{Z} \times
\mathbb{Z}/n\mathbb{Z})
\]
does not admit an unmixed ramification structure of size $(4,4)$.
\end{prop}
\begin{proof}
Observe that $G$ can be presented by
\[
    G = \langle t,r,s: t^2,r^m,s^n,[r,s],tsts,trtr \rangle,
\]
and so any element in $G$ can be written uniquely as
$t^{\epsilon}r^is^j$ for $\epsilon\in\{0,1\}$, $1\leq i \leq m$,
$1\leq j \leq n$.

Conjugation of some element $tr^is^j$ by $r^{-1}$ yields
$r^{-1}tr^is^jr = ttr^{-1}tr^is^jr = trr^irs^j = tr^{i+2}s^j$, and
similarly conjugation by $s^{-1}$ yields $tr^{i}s^{j+2}$. Hence,
$tr^is^j$ can be conjugated to $tr^{i+2k}s^{j+2l}$ for any $k,l$.

Let $A \cong \mathbb{Z}/m\mathbb{Z} \times \mathbb{Z}/n\mathbb{Z}$
be the maximal normal abelian subgroup in $G$, then $G \setminus A$
contains at most four conjugacy classes, represented by
$t,tr,ts,trs$. In fact, it contains one conjugacy class if both
$m,n$ are odd, two conjugacy classes if one of $m,n$ is odd and the
other is even, and four if both $m,n$ are even.

Since any spherical $4-$system of generators of $G$ must contain an
element of $G \setminus A$, the condition that $\Sigma_1 \cap
\Sigma_2 = \{1\}$ immediately implies that $m,n$ cannot both be odd.

Assume now that $m$ is even and $n$ is odd, and consider the
following map
\[
    G \twoheadrightarrow (\mathbb{Z}/2\mathbb{Z})^2, \text{ defined
    by } (t^{\epsilon}r^is^j)\mapsto \bigl(\epsilon, i (\text{mod } 2)\bigr).
\]
Note that for any $j$ and any odd $i$, one has that
\[
    (r^is^j)^{nm/2} =  (r^{m/2})^{ni}(s^n)^{mj/2} = r^{m/2} =:u \neq
    1.
\]

If $T=(t^{\epsilon_1}r^{i_1}s^{j_1}, t^{\epsilon_2}r^{i_2}s^{j_2},
t^{\epsilon_3}r^{i_3}s^{j_3}, t^{\epsilon_4}r^{i_4}s^{j_4})$ is a
spherical $4-$system of generators of $G$, then
$\epsilon_1+\epsilon_2+\epsilon_3+\epsilon_4 \equiv 0 \pmod 2$,
$i_1+i_2+i_3+i_4\equiv 0 \pmod 2$, and the images of the elements in
$T$ generate $(\mathbb{Z}/2\mathbb{Z})^2$. Hence, the image in
$(\mathbb{Z}/2\mathbb{Z})^2$ of such a spherical $4-$system of
generators can be (up to a permutation) only one of
\begin{enumerate}\renewcommand{\theenumi}{\it \roman{enumi}}
\item $\{(1,1),(1,0),(0,1),(0,0)\}$,
\item $\{(1,1),(1,1),(1,0),(1,0)\}$,
\item $\{(1,1),(1,1),(0,1),(0,1)\}$,
\item $\{(1,0),(1,0),(0,1),(0,1)\}$.
\end{enumerate}

Therefore, for any two spherical $4$-systems $T_1$ and $T_2$ one can
find $x \in T_1$ and $y \in T_2$ such that either
\begin{itemize}
\item $x,y \in G\setminus A$ are conjugate; or
\item $x,y \in A$ and $x^{mn/2}=u=y^{mn/2}$;
\end{itemize}
a contradiction to $\Sigma_1 \cap \Sigma_2 = \{1\}$.

If both $m$ and $n$ are even, write $m=2^{\mu}m'$ and $n=2^{\nu}n'$,
where $m',n'$ are odd. Without loss of generality, we may assume
that $\mu \geq \nu$.

Consider the following map
\[
    G \twoheadrightarrow (\mathbb{Z}/2\mathbb{Z})^3, \text{ defined
    by } (t^{\epsilon}r^is^j)\mapsto \bigl(\epsilon, i (\text{mod } 2), j (\text{mod } 2)\bigr).
\]
If  $\mu>\nu$ and if $i$ is odd then
\[
    (r^is^j)^{2^{\mu-1}m'n'} = (r^{2^{\mu-1}m'})^{in'}(s^{2^{\mu-1}n'})^{jm'}=r^{m/2}:=u \neq 1,
\]
and if $\mu=\nu$ then
\[
    (r^is^j)^{2^{\mu-1}m'n'} = \begin{cases}
        r^{m/2}:=u \neq 1, &\text{ if } i \text{ is odd and } j \text{ is even}; \\
        s^{n/2}:=v \neq 1, &\text{ if } i \text{ is even and } j \text{ is
        odd}; \\
        r^{m/2}s^{n/2}:=w \neq 1, &\text{ if } i,j \text{ are odd}.
    \end{cases}
\]

If $T=(t^{\epsilon_1}r^{i_1}s^{j_1}, t^{\epsilon_2}r^{i_2}s^{j_2},
t^{\epsilon_3}r^{i_3}s^{j_3}, t^{\epsilon_4}r^{i_4}s^{j_4})$ is a
spherical $4-$system of generators of $G$, then
$\epsilon_1+\epsilon_2+\epsilon_3+\epsilon_4\equiv 0 \pmod 2$,
$i_1+i_2+i_3+i_4\equiv 0 \pmod 2$, $j_1+j_2+j_3+j_4\equiv 0 \pmod 2$
and the images of the elements in $T$ generate
$(\mathbb{Z}/2\mathbb{Z})^3$. Hence, the image in
$(\mathbb{Z}/2\mathbb{Z})^3$ of such a spherical $4-$system of
generators can be (up to a permutation) only one of
\begin{enumerate}\renewcommand{\theenumi}{\it \roman{enumi}}
\item $\{(1, 1, 1),(1, 0, 0),(0, 1, 0),(0, 0, 1)\}$,
\item $\{(1, 1, 0),(1, 0, 1),(0, 1, 0),(0, 0, 1)\}$,
\item $\{(1, 1, 0),(1, 0, 0),(0, 1, 1),(0, 0, 1)\}$,
\item $\{(1, 1, 1),(1, 0, 1),(0, 1, 1),(0, 0, 1)\}$,
\item $\{(1, 0, 1),(1, 0, 0),(0, 1, 1),(0, 1, 0)\}$,
\item $\{(1, 1, 1),(1, 1, 0),(0, 1, 1),(0, 1, 0)\}$,
\item $\{(1, 1, 1),(1, 1, 0),(1, 0, 1),(1, 0, 0)\}$.
\end{enumerate}

Therefore, for any two spherical $4$-systems $T_1$ and $T_2$ one can
find $x \in T_1$ and $y \in T_2$ such that either
\begin{itemize}
\item $x,y \in G\setminus A$ are conjugate; or
\item $x,y \in A$ and $x^{2^{\mu-1}m'n'}=u=y^{2^{\mu-1}m'n'}$; or
\item $x,y \in A$ and $x^{2^{\mu-1}m'n'}=v=y^{2^{\mu-1}m'n'}$;
\end{itemize}
a contradiction to $\Sigma_1 \cap \Sigma_2 = \{1\}$.
\end{proof}

\begin{rem}\label{rem.semi.abelian}
In fact, the same argument also shows that the finite group $G=
\mathbb{Z}/2\mathbb{Z} \ltimes (\mathbb{Z}/m\mathbb{Z} \times
\mathbb{Z}/n\mathbb{Z})$ ($m,n \in \mathbb{Z}$) cannot admit an
unmixed ramification structure $(T_1,T_2)$, where $T_1$ has type
$(2,2,2,2)$.

Recall that $G \setminus A$ contains at most four conjugacy classes,
more precisely, it contains one conjugacy class if both $m,n$ are
odd, two conjugacy classes if one of $m,n$ is odd and the other is
even, and four if both $m,n$ are even.

Since any spherical system of generators of $G$ must contain an
element of $G \setminus A$, the condition that $\Sigma_1 \cap
\Sigma_2 = \{1\}$ immediately implies that $m,n$ cannot both be odd.

If $m$ is even and $n$ is odd, then the above argument shows that
the image in $(\mathbb{Z}/2\mathbb{Z})^2$ of any spherical system of
generators contains at least two of $(1,1),(1,0),(0,1)$, and hence
one can find $x \in T_1$ and $y \in T_2$ such that either $x,y \in
G\setminus A$ are conjugate, or $x,y \in A$ and $x^{mn/2}=y^{mn/2}$,
a contradiction to $\Sigma_1 \cap \Sigma_2 = \{1\}$.

If $m$ is even and $n$ is even, then the elements of order two in
$G$ are exactly $tr^is^j$ ($1\leq i \leq m$, $1\leq j \leq n)$,
$u=r^{m/2}$, $v=s^{n/2}$ and $w=r^{m/2}s^{n/2}$. The above argument
shows that $T_1$ either contains four elements from four different
conjugacy classes of $G\setminus A$, or two elements from two
different conjugacy classes of $G\setminus A$ and two different
elements of $\{u,v,w\}$. Since $T_2$ is also a spherical system of
generators, then one can find $x \in T_1$ and $y \in T_2$ such that
either $x,y \in G\setminus A$ are conjugate, or $y^i=x \in
\{u,v,w\}$, a contradiction to $\Sigma_1 \cap \Sigma_2 = \{1\}$.
\end{rem}


\end{document}